\documentclass[11pt,a4paper,reqno]{amsart}
\usepackage[T1]{fontenc}

\usepackage{amsfonts,amsmath,graphics,epsfig,latexsym,times}
\usepackage{ae,aecompl}
\usepackage[all]{xy}

\newcommand{\eq}[1][r]
   {\ar@<-3pt>@{-}[#1]
    \ar@<-1pt>@{}[#1]|<{}="gauche"
    \ar@<+0pt>@{}[#1]|-{}="milieu"
    \ar@<+1pt>@{}[#1]|>{}="droite"
    \ar@/^2pt/@{-}"gauche";"milieu"
    \ar@/_2pt/@{-}"milieu";"droite"}
\newif\ifpix \pixtrue
\usepackage[colorlinks=true]{hyperref}
\hypersetup{ pdfstartview=FitH }

\numberwithin{equation}{section}
\def\ZZ{{\mathbb{Z}}}    \def\QQ{{\mathbb{Q}}} \def\CC{{\mathbb{C}}}
\def\RR{{\mathbb{R}}}    
\def\cO{{\mathcal{O}}}

\def\cX{{\mathcal{X}}}
\def\cXhat{{\widehat {\mathcal{X}}}}
\def\cM{{\mathcal{M}}}

\def\cMhat{{\widehat {\mathcal{M}}}}

\def\Uhat{{\widehat {{U}}}}

\def\cK{{\mathcal{K}}}

\def\cY{{\mathcal{Y}}}
\def\cH{{\mathcal{H}}}

\def\cD{{\mathcal{D}}}

\def\cN{{\mathcal{N}}}

\def\cYhat{\widehat{\mathcal{Y}}}
\def\cNhat{\widehat{\mathcal{N}}}

\def\fh{{\mathfrak{h}}}
\def\fg{{\mathfrak{g}}}

\renewcommand{\epsilon}{\varepsilon}

\newcommand{\SU}{\mathrm{SU}}

\newcommand{\scal}{\mathrm{scal}}

\newcommand{\id}{\mathrm{id}}
\newcommand{\U}{\mathrm{U}}

\newcommand{\Ric}{\mathrm{Ric}}

\newcommand{\RP}{\mathbb{RP}}
\newcommand{\CP}{\mathbb{CP}}

\newcommand{\del}{\partial}
\newcommand{\delb}{\bar\partial}

\renewcommand{\Re}{\mathrm{Re}}

\newcommand{\ovom}{\overline{\omega}}

\newcommand{\ombar}{\ovom}

\newcommand{\omod}{\omega^{\mathrm{mod}}}

\newcommand{\Imod}{\mathrm{I}}
\newcommand{\Ieuc}{{\mathrm{euc}}}

\newcommand{\gmod}{\mathrm{g}}

\newcommand{\geuc}{{g^{\mathrm{euc}}}}

\newcommand{\Xbar}{{\overline{X}}}

\newcommand{\obar}{\overline{\omega}}
\newcommand{\gbar}{\overline{g}}

\newcommand{\Xhat}{{\widehat{X}}}

\newcommand{\Sbar}{\overline{\Sigma}}

\newcommand{\hkto}{\hookrightarrow}
\newcommand{\setC}{\mathbb{C}}
\newcommand{\setR}{\mathbb{R}}
\newcommand{\setZ}{\mathbb{Z}}
\newcommand{\setQ}{\mathbb{Q}}
\newcommand{\orb}{\mathrm{orb}}
\newcommand{\prs}{\mathfrak{R}^+}

\newcommand{\ordre}{\mathbf{d}}
\newtheorem{lemma}{Lemma}
\newtheorem{prop}[lemma]{Proposition}
\newtheorem{cor}[lemma]{Corollary}
\newtheorem{theo}[lemma]{Theorem}
\newtheorem{theointro}{Theorem}

\theoremstyle{definition}
\newtheorem{dfn}[lemma]{Definition}
\newtheorem{rmk}[lemma]{Remark}

\newtheorem*{exa}{Example}

\newtheorem*{rmk*}{Remark}
\newtheorem*{rmks*}{Remarks}

\title[Smoothing extremal surfaces]{Smoothing singular extremal K\"ahler surfaces and minimal Lagrangians}

\author{Olivier Biquard}
\address{Olivier Biquard, UPMC Universit\'e Paris 6 et \'Ecole Normale Sup\'erieure} 
\email{olivier.biquard@upmc.fr}

\author{Yann Rollin}
\address{Yann Rollin, Laboratoire Jean Leray, Universit\'e de Nantes}
\email{yann.rollin@univ-nantes.fr}

\begin{document}
{\Huge \sc \bf\maketitle}
\begin{abstract}
  Given a complex surface $\cX$ with singularities of class $T$ and no
  nontrivial holomorphic vector field, endowed with a K\"ahler class $\Omega_0$,
  we consider smoothings $(\cM_t,\Omega_t)$, where $\Omega_t$ is a K\"ahler class on
  $\cM_t$ degenerating to $\Omega_0$. Under an hypothesis of non degeneracy of
  the smoothing at each singular point, we prove that if $\cX$ admits an
  extremal metric in $\Omega_0$, then $\cM_t$ admits an extremal metric in $\Omega_t$
  for small $t$.
  
  In addition, we construct small Lagrangian stationary spheres which
  represent Lagrangian vanishing cycles when $t$ is small.
\end{abstract}

\section{Introduction}
\label{sec:intro}
Let $\cX$ be a normal compact complex surface with singularities of
class T. Such singularities are isolated and of orbifold type $\setC^2/\Gamma$
for some finite group $\Gamma\subset\U_2$. The possible singularities are either
rational double points ($\Gamma\subset\SU_2$) or cyclic quotient singularities of
type $\frac 1{dn^2}(1,dna-1)$, for positive integers $d$, $n$, $a$ such that
$a$ and $n$ are relatively prime.

Since there is no general answer to the existence problem of constant
scalar curvature K\"ahler metrics (CSCK for short), numerous
constructions relying on gluing techniques have been made (cf. for
instance \cite{RolSin05, ArePac06, ArePac09}). In the case of
canonical singularities (rational double points), the idea is to start
from a CSCK orbifold metric on $\cX$ and to deduce a smooth CSCK
metric on the minimal resolution $\cXhat$ by perturbation theory.  The
exceptional divisor of $\cXhat\to \cX$ is a union of holomorphic $-2$
spheres, whose configuration is described by the Dynkin diagram of a
complex semisimple Lie algebra $\fg_\setC$ determined by $\Gamma$ (this is a
part of the McKay correspondence).

In this work we consider smoothings of singularities rather than
resolutions. We prove, under natural hypotheses, that
$\setQ$-Gorenstein smoothings admit CSCK metrics.

In the special case of K\"ahler-Einstein metrics, the desingularizations
cannot carry K\"ahler-Einstein metrics unless $c_1(\cX)=0$, since they
contain holomorphic $-2$ spheres.
But smoothings may carry K\"ahler-Einstein metrics, and our result gives
a construction of those K\"ahler-Einstein metrics close to the singular
ones. Also it is interesting to note that for general singularities of
class T, the smoothings are not diffeomorphic to the minimal resolution.

If the homology class of a $-2$ sphere of the desingularization is
Lagrangian in the smoothing, we prove that it can be represented by a small
Hamiltonian stationary sphere. This gives a concrete construction of the
Lagrangian minimizer, whose existence is proved by Schoen and
Wolfson \cite{SchWol01}.

\subsection{CSCK metrics}
We now give precise statements.  Let $\cX$ be a normal complex surface with
quotient singularities.  We consider a $\setQ$-Gorenstein smoothing of $\cX$,
denoted $p:\cM\to\Delta$, where $\Delta$ is an open disc centered at the origin in
$\setC$. This means that a multiple of the canonical divisor of $\cM$ is
Cartier and $\cM$ is Cohen-Macaulay. Moreover, the central fiber is the
given $\cX$, and the general fiber $\cM_t$ is smooth.

K\'ollar and Shepherd-Barron \cite{KS88} proved that the singularities
which appear must be of class $T$; this class subdivides into two classes:
\begin{itemize}
\item canonical singularities of type $\setC^2/\Gamma$ for a finite subgroup
  $\Gamma\subset\SU_2$; one gets the list $A_d$ ($d\geq 1$), $D_d$ ($d\geq 4$), $E_d$
  ($d=6, 7, 8$);
\item cyclic quotients obtained by a quotient of a $A_{dn-1}$
  singularity by $\setZ_n$.
\end{itemize}

The local theory of smoothings of such a singularity is well
understood: there is an (explicit) hypersurface $\cY \subset \setC^3\times\setC^d$, such
that the projection $p:\cY\to\setC^d$ is a $\setQ$-Gorenstein smoothing of
$\cY_0\simeq \setC^2/\Gamma$. Moreover, any $\setQ$-Gorenstein smoothing of $\cY_0$ is
isomorphic to the pull-back of $p$ under a germ of holomorphic maps
$f:\setC\to\setC^d$.  The fibers $\cY_t=p^{-1}(t)$ over the discriminant locus
$\cD\subset \setC^d$ are singular. This leads us to introduce a non degeneracy
condition in the following way. As recalled below, there is a
particular ramified covering map $\pi:\setC^d\to\setC^d$ associated to the
singularity. This map is ramified exactly above $\cD$, and moreover
the ramification locus $\pi^{-1}(\cD)$, which turns out to be a union of
linear hyperplanes:
$$ \pi^{-1}(\cD) = \cup H_i. $$
Up to taking a ramified covering (which we choose of smallest possible
degree), the germ $f:\setC\to\setC^d$ can be lifted to
$\hat f:\setC\to\setC^d$ such that $f=\pi\circ \hat f$. The following definition is a
way to say that the deformation is transversal to the discriminant
locus:
\begin{dfn}\label{df:degenerate}
  We say that a smoothing $\cX\hkto \cM\to\Delta$ is non degenerate at the singular
  point $x\in \cX$ if the corresponding $\hat f$ satisfies $\frac{\partial\hat
    f}{\partial t}(0) \notin \cup H_i$.
\end{dfn}
\begin{exa}
  Consider the $A_k$-singularity $w^{k+1}=xy$, desingularized by the
  family
  $$ xy=w^{k+1}+a_{k-1}w^{k-1}+a_{k-2}w^{k-2}+\cdots +a_0.$$
  A deformation is given by a map $f(z)=(a_0(z),\ldots,a_{k-1}(z))$. It is easy to
  check that if $\frac{\partial a_0}{\partial z}\neq 0$ then the deformation is non degenerate.
\end{exa}

To state the theorem, we need to fix a K\"ahler class. Along a ray to the
origin in $\Delta$, the Gau\ss-Manin connection identifies $H^2(\cM_t,\setR)$ with a
fixed vector space; going to the origin, we identify this vector space with
$$H^2_\orb(\cX,\setR)\oplus \underset{x \text{ singular}}{\oplus}\setR^{e_x},$$
where $e_x$ is the dimension of the real $H^2$ of the local smoothing of
the singularity at $x$. As we shall see later, if $d_x$ is the
dimension of the space of smoothings of the singularity at $x$,
$$ e_x = \begin{cases} d_x & \text{for a canonical singularity,} \\ d_x-1 &
  \text{for other singularities.} \end{cases} $$
(One cannot get such an identification on the whole disk, because of the
monodromy of the Gau\ss-Manin connection; but since the data of a K\"ahler
class is real, it is natural to give it only on a ray).

The simplest form of our results can now be stated:
\begin{theointro}\label{th:CSCK1}
  Suppose that we are given
  \begin{itemize}
  \item a normal complex surface $\cX$, with no holomorphic vector
    fields, and a $\setQ$-Gorenstein smoothing $\cX\hkto\cM\to\Delta$, which is non
    degenerate at each singular point;
  \item along the ray $t\in \setR_+\cap \Delta$, a K\"ahler class
    $\Omega_t\in H^2(\cM_t,\setR)$, 
    depending smoothly on $t$, such that $\Omega_0$ is an orbifold K\"ahler class
    on $\cX$, containing an orbifold CSCK metric.
  \end{itemize}
  Then for small $t>0$ the class $\Omega_t$ contains a CSCK metric on $\cM_t$.
\end{theointro}
The behavior of the CSCK metric for $t$ small is well understood: outside
the singularities it converges to the orbifold CSCK metric on $\cX$, but
near a singularity $x$ some rescaling converges to an ALE K\"ahler Ricci flat
space, which appears as the `bubble' at $x$.

One special case is when the smoothing is given with a polarization $\Omega$,
resulting in a constant $\Omega_t$. This covers in particular the
K\"ahler-Einstein case: when $\Omega_t=\pm c_1(\cM_t)$ our result gives a
construction of the K\"ahler-Einstein metric on the smoothing of $\cX$, as
well as a concrete description of its degeneracy to an orbifold
K\"ahler-Einstein metric. This K\"ahler-Einstein picture is recently obtained
in the case of $A_1$ singularities by Spotti \cite{SpX12}.

Complex orbifold singularities of codimension at least $3$ are rigid by a
result of Schlessinger.  However, it would be interesting to extend our
results in dimension $3$ or more, allowing orbifold singularities of
codimension $2$. A natural problem would be to remove the 
triviality assumption of the automorphism group: one may expect the existence of the CSCK metric
to be related to a K-stability property, like in \cite{SzX10}.

\medskip
One can replace the non degeneracy hypothesis on the smoothing $\cM_t$ by a
weaker non degeneracy hypothesis on the data $(\cM_t,\Omega_t)$ of the smoothing
with the K\"ahler class. This requires slightly more material that we now
explain.

Let us come back to the $\setQ$-Gorenstein smoothing $\cY\to\setC^d$ of a canonical
singularity. The parameter space $\setC^d$ can be identified with $\fh_\setC/W$,
where $\fh_\setC$ is a Cartan subalgebra of the Lie algebra $\fg_\setC$ associated
to $\Gamma$ by the McKay correspondence ($\fg_\setC$ is exactly the $A_d$, $D_d$ or
$E_d$ simple Lie algebra), and $W$ is the Weyl group. Here the
canonical projection $\fh_\CC\to\CC^d$ is ramified over the
discriminant locus $\cD\subset\CC^d$. The real 2-cohomology
is isomorphic to the real Cartan subalgebra $\fh_\setR$ (actually, the
smoothing is diffeomorphic to the minimal resolution, whose $-2$ spheres
give a basis of $(\fh_\setR)^*$). But there is some ambiguity in this
identification, because the monodromy of the Gau\ss-Manin connection is given
by the action of $W$.

In the smoothing $\cY\to\setC^d=\fh_\setC/W$, the general fiber is smooth, but there
are singular fibers, given by the walls of the Weyl chambers of
$\fh_\setC$. Nevertheless, there exist a simultaneous resolution of
singularities $\cYhat\to\fh_\setC$ (Brieskorn, Slodowy). Then all the fibers are
smooth and their real 2-cohomology can be identified to~$\fh_\setR$.

The case of the other T singularities is similar: they are quotients of
$A_{dn-1}$ singularities by a $\setZ_n$ action, and one obtains the same
structure by taking the $\setZ_n$ fixed points of the $A_{dn-1}$ simultaneous
resolution: the space of parameters is therefore $\fh_\setC^{\setZ_n}\simeq \setC^d$, and
the real 2-cohomology is $\fh_\setR^{\setZ_n}\simeq \setR^{d-1}$.

Now come back to our setup of a smoothing $\cM\to\Delta$ with a family of K\"ahler
classes $\Omega_t$. At a singular point $x\in \cX$, the smoothing is induced from
the standard smoothing $\cY$ by a map (we denote the dependence in $x$ by
an index $x$)
$$ f_x:\Delta_c \longrightarrow ((\fh_\setC)_x/W_x)^{G_x}, $$
where $G_x=1$ for a canonical singularity, and $\Delta_c\subset\Delta$ is a smaller disk of
radius $c$. Up to a finite ramified covering $\Delta_b\to\Delta_c$, we can lift this map into a map into
$(\fh_\setC)_x^{G_x}$:
$$ (\zeta_c)_x:\Delta_b \longrightarrow (\fh_\setC)_x^{G_x}. $$

Up to taking a higher degree covering, we can take the same covering
$\Delta_b\to\Delta_c$, say of order $\ordre$, for all singular points. Adding the
data $(\zeta_r)_x$ of the K\"ahler class in a neighborhood of the
singularity, we obtain at each singularity a map
$$ \zeta_x:\Delta_b \longrightarrow (\fh_\setR)_x^{G_x} \oplus (\fh_\setC)_x^{G_x} , \quad \zeta_x=((\zeta_r)_x,(\zeta_c)_x), $$
which represents the deformation of both the K\"ahler parameter and the
complex parameter. We consider the K\"ahler class as depending smoothly on the
parameter on $\Delta_b$ (since $\Delta_b$ is a covering, this is more general than
the setting of Theorem~\ref{th:CSCK1}). Then at the origin we have for some
integer $p$
$$ \zeta_x(t) = t^p \dot \zeta_x + O(t^{p+1}). $$
We can now define an extended notion of non degeneracy:
\begin{dfn}\label{df:degenerate2}
  We say that $(\cM_t,\Omega_t)$ is non degenerate at the point $x$ if $p\leq
  \ordre$ and $\dot \zeta_x$ does not belong to a wall of a Weyl chamber.
\end{dfn}

The above definition is clearly a generalization of
Definition~\ref{df:degenerate}.  We now extend Theorem~\ref{th:CSCK1}
under the form:
\begin{theointro}\label{th:CSCK2}
  Suppose that we are given
  \begin{itemize}
  \item a normal complex surface $\cX$, with no holomorphic vector
    fields, and a $\setQ$-Gorenstein smoothing $\cX\hkto\cM\to\Delta$;
  \item along the ray $t\in \setR_+\cap \Delta_b$, a K\"ahler class $\Omega_t\in H^2(\cM_t,\setR)$, such
    that $\Omega_0$ is an orbifold K\"ahler class on $\cX$, containing an
    orbifold CSCK metric.
  \end{itemize}
  If $(\cM_t,\Omega_t)$ is non degenerate at each singular point in the sense of
  Definition~\ref{df:degenerate2}, then for small $t>0$ the class $\Omega_t$
  contains a CSCK metric on $\cM_t$.
\end{theointro}

There is a final extension of the results in the case of canonical
singularities: then we do not need the initial family $\cM$ to be a
smoothing. Indeed suppose that $\cX\hookrightarrow\cM\to\Delta$ is any deformation of a surface
with rational double points, then near each singularity $x$, the family is
still induced by some map $\Delta_c\to(\fh_\setC)_x/W_x$ (because this is a
semi-universal deformation), which can again be lifted to
$(\zeta_c)_x:\Delta_b\to(\fh_\setC)_x$. Pulling back $\cYhat\to(\fh_\setC)_x$ near each singular
point, we obtain a simultaneous resolution $\cMhat\to\Delta_b$ of the
singularities of $\cM\to\Delta$ and we can apply our method in this setting to
obtain:
\begin{theointro}\label{th:CSCK3}
  Suppose that we are given
  \begin{itemize}
  \item a normal complex surface $\cX$, with no holomorphic vector
    fields, with only rational double points, and a deformation
    $\cX\hkto\cM\to\Delta$; let then $\cMhat\to\Delta_b$ the simultaneous resolution of
    singularities as above;
  \item along the ray $t\in \setR_+\cap \Delta_b$, a K\"ahler class $\Omega_t\in H^2(\cMhat_t,\setR)$,
    such that $\Omega_0$ is an orbifold K\"ahler class on $\cX$, containing an
    orbifold CSCK metric.
  \end{itemize}
  If $(\cMhat_t,\Omega_t)$ is non degenerate at each singular point in the sense
  of Definition~\ref{df:degenerate2}, then for small $t>0$ the class $\Omega_t$
  contains a CSCK metric on $\cMhat_t$.
\end{theointro}
This result covers the known case of minimal resolutions, but also the case
of partial resolutions which are not smoothings. It is not valid for
$T$ singularities because the deformations include other components that
are not induced from the $\setQ$-Gorenstein smoothing $\cY$.

\subsection{Hamiltonian stationary spheres}

Suppose we are under the hypothesis of Theorem~\ref{th:CSCK2}. As we have
seen,
$$ H^2(\cM_t,\setR) = H^2_\orb(\cX,\setR) \oplus
\oplus_{x \text{ singular}} (\fh_\setR)_x^{G_x}. $$ Fix a singular point $x$, and a
system $\prs_x\subset(\fh_\setC)_x^*$ of positive roots. If $x$ is a rational double
point ($G_x=1$), a root $\theta\in \prs_x$ represents an integral homology class
on $\cM_t$; in general, $\theta$ represents a homology class in the local
$G_x$-covering, so after projection still represents a (possibly zero)
homology class in $\cM_t$.

Remind that the deformation $(\cM_t,\Omega_t)$ is given near the singular point
$x$ by a map $\zeta_x=(\zeta_r)_x+(\zeta_c)_x:\Delta_b\to(\fh_\setR)_x^{G_x} \oplus (\fh_\setC)_x^{G_x}$,
in which $(\zeta_r)_x(t)$ represents the class of the restriction of $\Omega_t$ near
the singularity. Therefore the homology class defined by $\theta$ is
$\Omega_t$-Lagrangian if $\langle\theta,(\zeta_r)_x(t)\rangle\equiv 0$, which implies in particular $\langle\theta,
(\dot \zeta_r)_x\rangle=0$.
\begin{theointro}\label{th:HSS}
  Suppose that we are under the hypothesis of Theorem~\ref{th:CSCK2} or
  \ref{th:CSCK3}. Fix a singular point $x\in \cX$ and a root $\theta\in \prs_x$, such
  that $\langle\theta,(\dot \zeta_r)_x\rangle=0$, and $\theta$ is primitive for this property (that is
  cannot decompose as $\theta_1+\theta_2$ with $\theta_i\in \prs_x$ satisfying the same
  property). Finally suppose that for all $t>0$, one has $\langle\theta,(\zeta_r)_x(t)\rangle\equiv 0$,
  that is the homology class represented by $\theta$ remains $\Omega_t$-Lagrangian.
  Then:
  \begin{enumerate}
  \item If $x$ is a rational double point, then the homology class in
    $\cM_t$ (or $\cMhat_t$ in the setting of Theorem~\ref{th:CSCK3})
    corresponding to $\theta$ is represented by a smooth Hamiltonian stationary
    sphere $S_t$, which is also a global Lagrangian minimizer of the
    area in its homotopy class.
  \item If $G_x=\setZ_n$, and if the homology class defined by $\theta$ is nonzero,
    then it is represented by a smooth Hamiltonian stationary sphere $S_t$.
  \end{enumerate}
\end{theointro}
Moreover (see section \ref{sec:hamilt-stat-spher}), one gets an explicit
description when $t\to0$: after dilations, the Hamiltonian stationary
representative of $\theta$ converges to a special Lagrangian sphere in the
ALE K\"ahler Ricci flat space obtained at the limit.

In the second case ($G_x=\setZ_n$), if the homology class defined by $\theta$ in the
quotient vanishes, one can still get in some cases a Hamiltonian stationary
embedded $S^2$ or $\RP^2$, see section~\ref{sec:hamilt-stat-spher-1}. We do
not state a minimizing property which is less clear, because the models in
the ALE K\"ahler Ricci flat space are not calibrated.

For example, from Theorem~\ref{th:CSCK1} we recover K\"ahler-Einstein metrics
on certain $4$-point blowup of $\CP^2$ constructed by Tian \cite{Tia90},
and Theorem~\ref{th:HSS} provides a construction of stationary Lagrangian
spheres (cf \S\ref{sec:dp}).

\subsection{Organization of the paper}
In most of the article, we suppose that $\cX$ has only canonical
singularities and we prove directly Theorem~\ref{th:CSCK3}, which implies
Theorem~\ref{th:CSCK1}, Theorem~\ref{th:CSCK2}) and Theorem~\ref{th:HSS}. At the
end, we explain the case of general singularities of the class T: there is
not much change, because they are obtained as cyclic quotients of canonical
singularities and all our constructions pass to the quotient.

One of the main point in the paper is the construction of `good' K\"ahler
metrics on the deformations: this is not obvious, because the complex
structure changes. It turns out that what is needed is an extension to the
singular setting of the theorem of Kodaira on the stability of the K\"ahler
condition under complex deformations. This is done in section
\ref{sec:rep}, after the introduction of the `tangent graviton' in section
\ref{sec:defo}. In section \ref{sec:csck-metrics}, we extend arguments in
the literature to produce the CSCK metrics, so we only point out the new
features in our situation. In section \ref{sec:hamilt-stat-spher}, we
produce the Hamiltonian stationary Lagrangian spheres. The case of general
singularities is explained in section \ref{sec:gor}, and some applications
are given in section \ref{sec:appli}.
\subsection{Acknowledgements}
The authors would like to thank Claude Le\,Brun for some stimulating
discussions.

\section{Deformations of singular surfaces and K\"ahler classes}
\label{sec:defo}
In this section, $\cX$ will always denote a compact complex surface with
canonical singularities endowed with an orbifold K\"ahler metric $\gbar$ with
K\"ahler form $\obar$ and K\"ahler class $\Omega_0$.
We consider a flat deformation $\cX\hkto\cM\to\Delta$ of $\cX$.

\subsection{Complex deformations near singularities}
\label{sec:family}

Our first goal is to understand a simultaneous resolution $\cXhat\hookrightarrow
\cMhat\to\Delta_{d}$ of the deformation, focusing near the singular points. For
this purpose, we choose a set of adapted isomorphisms $U_i\simeq \Delta^2_{\epsilon_i}/\Gamma_i$
near each singularity $x_i\in \cX$. Actually, for simplicity of notations
we will suppose that there is only one singular point $x_0$, and we will
point out from place to place what changes for several singular points. So
near $x_0$, one can choose local coordinates $z$ with values in $\Delta^2_\epsilon$
defined modulo $\Gamma$, such that the pullback of $\obar$ on the disc is
expressed as
\begin{equation}
\label{eq:kahlloc}
dd^c (|z|^2+\eta) 
\end{equation}
where $|z|$ is the Euclidean distance to the origin in $\setC^2$ and $\eta$ is a
smooth $\Gamma$-invariant real function such that $\eta=\cO(|z|^4)$.  Up to scaling
the metric, and shrinking the open set $U$, we may assume that $\epsilon=1$.
In the sequel, we shall always assume that the identification $U\simeq
\Delta^2/\Gamma$ is chosen in such a way.

By restricting $\cM$ to a suitable neighborhood of the singular point $x_0$,
we deduce a flat deformation
$$\Delta^2/\Gamma\hkto \cN\to \Delta_c$$
of the singularity $U=\cN_0$.  The singularity $U$ admits a semi-universal
flat family of deformations by a result of Kas-Schlessinger
\cite{KasSch72}: there is a deformation
$$\setC^2/\Gamma\hkto \cY\to \fh_\setC/W, $$
where $\fh_\setC$ is a Cartan subalgebra of the complex semisimple Lie algebra
associated to the Dynkin diagram of the singularity (this is the Lie
algebra associated to the finite group $\Gamma$ by the McKay correspondence),
and $W$ is the Weyl group.  Then, $\cN$ is induced by some holomorphic map
$\psi:\Delta_c\to \fh_\setC/W$. Thus, there is a holomorphic commutative diagram
$$
\xymatrix{
\Delta^2/\Gamma\ar[d]\ar[r]&\setC^2/\Gamma\ar[d]\\
\cN\ar[d]\ar[r]_\Psi & \cY\ar[d]\\
\Delta_c\ar[r]_\psi & \fh_\setC/W
}
$$
such that the restriction $\Psi:\cN_t\to\cY_{\psi(t)}$ is an embedding
for every $t\in\Delta_c$.

A remarkable feature of canonical singularities is that
$\setC^2/\Gamma\hkto\cY\to\fh_\setC/W$ admits a simultaneous resolution
$\widehat{\setC^2/\Gamma} \hookrightarrow\cYhat\to \fh_\setC$ given by a diagram
\begin{equation} \label{eq:2}
\xymatrix{
\widehat{\setC^2/\Gamma}\ar[r]\ar[d]& {\setC^2/\Gamma}\ar[d]\\
\cYhat\ar[r]\ar[d]&\cY\ar[d]\\
\fh_\setC\ar[r]&\fh_\setC/W 
}
\end{equation}
where the map $\fh_\setC\to\fh_\setC/W$ is the canonical projection (Brieskorn, Slodowy).

A priori, one cannot lift the map $\psi:\Delta_c\to\fh_\setC/W$ to $\fh_\setC$. The
obstruction is the monodromy of $\psi$, which lies in $W$. By taking a
ramified cover, with order $\ordre$ equal to the order of the monodromy of $\psi$,
we obtain a lifting
\begin{equation}\label{ordre:revetement}
\xymatrix{
\Delta_d\ar[r]^{\hat\psi}\ar[d]_{r_\ordre} & \fh_\setC\ar[d]\\
\Delta_c\ar[r]^\psi & \fh_\setC/W
}\end{equation}
If there are several singular points, one has to take the order to be
the least common multiple of the orders at each point.

Thus, the family of deformations $\cN'=\Delta_d\times_{r_k}\cN$ admits a simultaneous
resolution $\cNhat\to \cN'$ and we have a commutative holomorphic diagram
\begin{equation}
  \label{eq:commlift}
\xymatrix{
\cYhat\ar[ddd] \ar[rrr]& & & \cY\ar[ddd]\\
&\cNhat\ar[ul]_{\hat\Psi}\ar[r]\ar[d] & \cN\ar[ur]^{\Psi}\ar[d] & \\
& \Delta_d \ar[r]^{r_k}\ar[dl]_{\hat\psi} & \Delta_c\ar[dr]^{\psi} &\\
\fh_\setC \ar[rrr]^\pi &&& \fh_\setC/W
}  
\end{equation}
where the maps $\Psi$ an $\hat\Psi$ restrict to fiberwise embeddings.  Here we
should point out that the maps also commute with the embeddings $\widehat
{\Delta^2/\Gamma}\hkto \cNhat$, $\Delta^2/\Gamma\hkto \cN$, $\widehat {\setC^2/\Gamma}\hkto
\cYhat$, $\setC^2/\Gamma\hkto \cY$ and the canonical inclusion $\Delta^2\hkto \setC^2$.

\begin{rmk}
  The fact that deformations of simple singularities do admit simultaneous
  resolutions after passing to a sufficiently high ramified cover, as
  recalled above, is the essential ingredient used in \cite{KasSch72} to
  construct a simultaneous resolution for deformations of compact complex
  surfaces with canonical singularities.
\end{rmk}

\begin{rmk}
  Conversely, if the simultaneous resolution $\cXhat\hkto\cMhat\to \Delta_d$ is
  already given as a data, the preimage $\cNhat$ of $\cN$ via the map
  $\cMhat\to \cM$ provides a simultaneous resolution of the deformation of
  the singularity for free.  However, by a universal property of
  simultaneous resolutions of canonical singularities~\cite{Hui73}, it
  turns out that $\widehat {\Delta^2/\Gamma}\hkto\cNhat\to \Delta_d$ must be given by the
  above construction.
\end{rmk}

\subsection{Kronheimer's gravitons}
\label{sec:krogra}

The simultaneous resolution $\cYhat\to\fh_\setC$ of the semi-universal family of
deformations $\cY\to\fh_\setC/W$ is explicitly constructed by Kronheimer.  At
this point, we need more details. Kronheimer actually constructs a family
$Y_\zeta$ of hyperK\"ahler manifolds, parameterized by a triple $\zeta=(\zeta_1,\zeta_2,\zeta_3)\in
\fh\otimes\setR^3$. To explain its properties, we choose a positive root system
$\prs\subset\fh^*$, and for a root $\theta\in\prs$ we define a hyperplane of $\fh$ by
$$ D_\theta = \ker \theta . $$
Then the family $(Y_\zeta)$ has the following properties:
\begin{enumerate}
\item $Y_\zeta$ is a smooth manifold if 
  \begin{equation}
\zeta\notin D=\cup_{\theta\in\prs} \setR^3\otimes D_\theta;\label{eq:4}
\end{equation}
\item all $(Y_\zeta)_{\zeta\notin D}$ are diffeomorphic to the minimal resolution
  $\widehat{\setC^2/\Gamma}$ of $\setC^2/\Gamma$; the diffeomorphism
  $$ F_\zeta: \widehat{\setC^2/\Gamma} \overset{\sim}{\longrightarrow} Y_\zeta $$
  can be chosen so that the metric is asymptotic to the Euclidean
  metric: actually there is an asymptotic development at any order
  $$ F_\zeta^* \gmod_\zeta = \geuc + \sum_{i=2}^{k-1} g^i_\zeta R^{-2i} +
  \cO(R^{-2k}), $$ where $g^i_\zeta$ is a homogeneous polynomial in $\zeta$ of
  degree $i$; in the sequel, we will suppose that such a
  diffeomorphism is chosen (this is possible smoothly in the $\zeta$
  parameter because $\fh\otimes\setR^3-D$ is simply connected);
\item $H^2(Y_\zeta,\setR)$ is identified with $\fh$ in such a way that the
  homology classes of the $-2$ curves of the resolution get identified
  with the simple roots of $\fh$ (and so $H_2(Y_\zeta,\setZ)$ is identified with
  the root lattice of $\fh$); under this identification, the cohomology
  classes of the three K\"ahler forms $(\omega_1,\omega_2,\omega_3)$ of $Y_\zeta$ are
  $\zeta_1$, $\zeta_2$ and $\zeta_3$;
\item there is a $SO_3$ action which for $u\in SO_3$ identifies
  isometrically
  $$ Y_\zeta \longrightarrow Y_{u(\zeta)},$$
  permuting $(\omega_1,\omega_2,\omega_3)$ to $u(\omega_1,\omega_2,\omega_3)$;
\item when we want to underline the holomorphic symplectic structure
  $(I_1,\omega_2+i\omega_3)$ of $Y_\zeta$, we use the notation $Y_{\zeta_r,\zeta_c}$, where
  $\zeta_r=\zeta_1$ and $\zeta_c=\zeta_2+i\zeta_3\in \fh_\setC$; then there is a $\setC^*$-action, giving
  for $\lambda\in \setC^*$ an isomorphism of holomorphic symplectic manifolds, which is
  actually also an isometry,
  \begin{equation}
    \label{eq:3}
    H_\lambda : Y_{\zeta_r,\zeta_c} \longrightarrow \frac1{|\lambda|^2}Y_{|\lambda|^2\zeta_r,\lambda^2\zeta_c} ;
  \end{equation}
Here the leading fraction means that the metric has been rescaled by a
factor $\frac 1{|\lambda|^2}$.
\item if $\zeta_r\notin \cup_{\theta\in R_+} D_\theta$, then there is a map $Y_{\zeta_r,\zeta_c} \to
  Y_{0,\zeta_c}$ which is a minimal resolution of singularities, actually
  $Y_{\zeta_r,\zeta_c}$ and $Y_{0,\zeta_c}$ are the fibers $\cYhat_{\zeta_c}$ and
  $\cY_{\zeta_c}$ of the simultaneous resolution (\ref{eq:2}).
\end{enumerate}

\begin{rmk}
  In this paper, the notation $\cO(R^k)$ for a function $f$ on $\setC^2$, or
  more generally a tensor, means that when $R$ goes to $\infty$, then for any
  integer $l=0,1,\ldots$, one has $\nabla^l f=\cO(R^{k-l})$.
\end{rmk}

We now point out the statement which will be central in this paper: it
gives the geometric meaning of the walls $D_\theta$.

\begin{lemma}\label{lem:lag-hol-cycle}
  Suppose $\theta$ is a root and $\zeta\notin D$, so $\theta$ corresponds to some
  homology class in $H_2(Y_\zeta,\setZ)$.

  1) If $\zeta_1\in \ker \theta$, then $\theta$ is a Lagrangian homology class for
  $Y_{\zeta_1,\zeta_c}$.

  2) If $\zeta_c\in \ker \theta$, then $\theta$ is represented by a holomorphic cycle
  in $Y_{\zeta_1,\zeta_c}$; moreover if $\theta$ is primitive for this property
  (that is cannot be written as $\theta_1+\theta_2$, with $\theta_i\in\prs$ and $\zeta_c\in
  \ker \theta_i$), then $\theta$ is represented by a holomorphic sphere.
\end{lemma}
The condition $\zeta\notin D_\theta$ can now be understood in terms of this lemma:
indeed, if $\zeta\in D_\theta$, then both $\zeta_r,\zeta_c\in \ker \theta$ which means that $\theta$
would represent at the same time a Lagrangian class and a holomorphic
cycle, which is impossible, so $Y_{\zeta_r,\zeta_c}$ has to be singular.

The first part of the lemma is obvious from (iii). The second part is
basically contained in the work of Brieskorn and Slodowy on Kleinian
singularities. By property (vi), the holomorphic map
$Y_{\zeta_r,\zeta_c}\to Y_{0,\zeta_c}=\cY_{\zeta_c}$ is a minimal resolution of
singularities, but the semi-universal deformation $\cY$ is explicit
and its singularities are completely understood: the discriminant
locus $\cD\subset\setC^k=\fh/W$, that is the set of $v\in \setC^k$ such that $\cY_v$
is singular, is exactly the branch locus of the projection
$\fh_\setC\to\fh_\setC/W$, i.e. the projection of the kernels of the roots.
(And the monodromy representation $\pi_1(\setC^k-\cD)\to\mathrm{Aut}\,H^2(\cY_{v_0},
\setC)$ is the standard representation of $W$ on $\fh_\setC$).

If $\zeta_c\in \ker \theta$ for only one root $\theta$ (this is the generic case),
then $\cY_{\zeta_c}$ has a singular point with a singularity of type
$\setC^2/\setZ_2$, giving a $-2$ holomorphic sphere in the minimal resolution
$Y_{\zeta_r,\zeta_c}$, in the homology class corresponding to $\theta$. For a
general $\zeta_c\in \ker \theta$, then of course $\theta$ is still represented by a
holomorphic cycle in $Y_{\zeta_r,\zeta_c}$, which might be a union of several
curves if it can be decomposed as a sum of roots $\theta=\theta_1+\cdots +\theta_\ell$ such
that $\zeta_c\in \ker \theta_i$.

\begin{rmk}
  The $A_k$-gravitational instantons are known explicitly
  (multi-Eguchi-Hanson metrics given by the Gibbons-Hawking
  ansatz). In that case one can see explicitly the holomorphic cycles
  of Lemma~\ref{lem:lag-hol-cycle}.
\end{rmk}

\subsection{The tangent graviton}
\label{sec:induced}
We come back to our setting of a flat deformation
$\cX\hkto\cM\to\Delta$ of a K\"ahler orbifold
surface $\cX$ with  a simultaneous
resolution $\cXhat\hkto\cMhat\to\Delta_d$ after passing to a ramified
cover. We deduce a family of deformations of the singularity
$\Delta^2/\Gamma\hkto\cN\to \Delta_c$ and a simultaneous resolution
$\widehat{\Delta^2/\Gamma}\hkto\cNhat\to \Delta_d$.
As explained in \S\ref{sec:family} we have a morphism
$$\xymatrix{
\cNhat\ar[d]\ar[r]^{\hat\Psi}& \cYhat\ar[d]\\
\Delta_d\ar[r]^{\hat\psi}&\fh_\setC
}
$$
The cohomology class $\Omega_t$ defined for $t\in \Delta_d\cap \RR^+$, restricted to $\cNhat_t$ defines  a
class on  $\cYhat_{\hat \psi(t)}$ identified to an element 
 $\zeta_r(t)\in \fh$. So that the whole data
$(\cYhat_{\hat \psi(t)},\zeta_r(t))$ is exactly that of the Kronheimer
graviton $Y_{\zeta_r(t),\zeta_c(t)}$ with the definition
$\zeta_c(t)=\hat \psi(t)$. 
In
view of Lemma~\ref{lem:lag-hol-cycle}, it is clear that the parameters
$(\zeta_r(t),\zeta_c(t))$ satisfy condition (\ref{eq:4}), so
$Y_{\zeta_r(t),\zeta_c(t)}$ is smooth.

The fact that $\Omega_t$ converges to an orbifold K\"ahler class $\Omega_0$ means that
$\zeta_r(t)$ converges to $0$, and we suppose that it depends smoothly on
the parameter $t$, including at $t=0$.

We assume that the map
\begin{align*}
\zeta:  \Delta_d\cap \RR^+ &\to \fh \oplus \fh_\setC\\
t& \mapsto (\zeta_r(t),\zeta_c(t))
\end{align*}
does not vanish at infinite order at $t=0$ and introduce the first
nonzero derivative $\dot\zeta$ for some order $p>0$.
This means that
\begin{equation}
  \label{eq:assnonvanish}
   \zeta(t)= t^p\dot\zeta+ O(t^{p+1}) 
\end{equation}
for some $p>0$ and $\dot\zeta\neq 0$.

The domain $\cNhat_t$ identifies to a small domain of $Y_{\zeta(t)}$. Zooming
by a factor $\epsilon^{-1}=t^{-p/2}$, and multiplying the K\"ahler class by a factor
$\epsilon^{-2}=t^{-p}$, we obtain by (\ref{eq:3}) that this domain is identified
via $H_{t^{-p/2}}$ with a larger and larger domain in $Y_{\epsilon^{-2}\zeta(t)}$,
which converges to $Y_{\dot \zeta}$ on compact subsets.

This discussion motivates:
\begin{dfn}
\label{dfn:tangra}
The Kronheimer space $Y_{\dot \zeta}$ is called the tangent graviton
to the deformation $(\cNhat,\Omega_t|_{\cNhat_t})$.
\end{dfn}

\begin{rmk}
  There are choices in the construction of $\dot \zeta$:
  \begin{itemize}
  \item the choice of lifting to the simultaneous resolution is done up to
    the action of the Weyl group $W$, acting on the parameter $\zeta_c(t)$: but
    this does not change the space $Y_{\zeta_r(t),\zeta_c(t)}$;
  \item the choice of a coordinate in the disc: since we have chosen a real
    ray, the ambiguity is the rescaling by a real $\lambda>0$, but in view of
    (\ref{eq:3}), this amounts to rescale the graviton.
  \end{itemize}
\end{rmk}

\begin{rmk}
  If the restriction of $\Omega_t$ to $\cNhat_t$ is identically zero, then
  $\zeta_r(t)\equiv 0$. In this case, as $\zeta$ is a holomorphic map and the nonzero
  derivative $\dot\zeta$ can be defined without restricting to a particular ray
  in $\Delta_d$. Thus, all the above definition make sense if we allow $t\in
  \Delta_d$. This property is of special interest in the case of polarized
  smoothings.
\end{rmk}

\section{Representing  K\"ahler classes} 
\label{sec:rep}
The one point blowup of a complex manifold endowed with a K\"ahler metric
$\omega$, carries a family of K\"ahler metric $\omega_\epsilon$. This is a very nice argument
due to Kodaira in which the metric $\omega_\epsilon$ can be constructed almost
explicitly. The construction of $\omega_\epsilon$ is done by cut and paste, where the
Burns-Simanca metric defined on a neighborhood of the exceptional divisor
is glued with the original metric. As $\epsilon \to 0$, the metrics $\omega_\epsilon$ converge
smoothly away from the exceptional divisor toward the original metric~$\omega$.

The aim of the current section is to prove a similar result for families of
deformations of a K\"ahler orbifold surface with isolated singularities.  In
particular, we shall prove the following proposition.
\begin{prop}
\label{prop:kahlmet}
Let $\cX\hkto\cM\to \Delta$ be a family of deformations of a compact complex
surface with canonical singularities.  Let $\ombar$ be an orbifold K\"ahler
metric on $\cX$ with K\"ahler class $\Omega_0$ and $\Omega$ a family of K\"ahler
classes supported by a simultaneous resolution $\cXhat\to\cMhat\to \Delta_d$
degenerating toward $\Omega_0$, such that the variation of K\"ahler class and
complex structure is non degenerate in the sense of 
Definition~\ref{df:degenerate2}.

Then, there exists a family of K\"ahler metrics $g_t$ with K\"ahler forms $\omega_t$
on $\cMhat_t$ defined for $t\in \Delta_d\cap (0,+\infty)$ and a smooth trivialization
$\phi:\Delta_d\times \cXhat \to \cMhat$ identifying all the fibers $\cM_t$ with the
property that
\begin{itemize}
\item $[\omega_t]=\Omega_t$;
\item the family of metrics $g_t$ converges in the $C^2$-sense toward the
  orbifold metric $\gbar$ on every compact set of $X=\cXhat\setminus E$, where $E$
  is the exceptional divisor of $\cXhat \to \cX$.
\end{itemize}
\end{prop}
\begin{rmk}
  The above proposition also holds if $\Omega$ is only assumed to be a family of
  $(1,1)$-classes (instead of K\"ahler). As a corollary, under the
  assumptions of admissibility, $\Omega_t$ must be a K\"ahler class for $t\in(0,+\infty)$
  sufficiently small.
\end{rmk}
% A famous result due to Kodaira states that the property of being
% K\"ahler is stable under complex deformations in the case of nonsingular manifolds.
% Proposition \ref{prop:kahlmet} is essentially a refinement of
% Kodairas's  result in a singular context. 

The rest of the section will be devoted to prove the proposition as well as
giving more accurate results about the behavior of $\omega_t$ as $t\to 0$.

\subsection{Summary of the setup}\label{sec:summary}
We use the notations introduced at \S\ref{sec:defo}. As in
\S\ref{sec:family}, we shall assume to keep the notations simple that there
is exactly one singularity in $\cX$ with a neighborhood $U$ identified to
$\Delta^2/\Gamma$. The minimal resolution of the singularity will be denoted
$\Uhat\to U$ and $\cXhat\to\cX$.  The case when several singularities occur
is a straightforward generalization of the constructions explained below.

The smooth manifold or orbifold deduced from $\cXhat$ and $\cX$ will be
denoted $\Xhat$ and $\Xbar$, and $X$ will denote the complement of the
singularity in $\Xbar$, or equivalently, the complement of the exceptional
divisor in $\Xhat$.

Using a smooth trivialization $\phi:\Delta_d\times \Xhat\to \cMhat$ of the simultaneous
resolution $\cXhat\hkto\cMhat\to\Delta_d$, we have a family of complex
structures $J_t$ defined on $\Xhat$ deduced from $\cMhat_t$ and $\phi$. The
degenerating family of K\"ahler classes $\Omega_t$ on $(\Xhat,J_t)$ also provides
a cohomology class $\Omega_t|_{\Uhat}\in H^2(\Uhat,\RR)\simeq
H^2(\widehat{\setC^2/\Gamma},\RR)$. As explained in \S\ref{sec:induced}, this is
the K\"ahler class of a K\"ahler Ricci-flat metric $\gmod_{\zeta(t)}$ compatible
with the complex structure $\Imod_{\zeta(t)}$ on $\widehat{\setC^2/\Gamma}$ provided by
Kronheimer's construction.

If the variation of the deformation is non degenerate, there is a scaling
parameter $\epsilon=t^{p/2}$ defined for some $p\geq 1$, and homotheties $H_{\epsilon^{-1}}$
by a factor $t^{-p/2}=\epsilon^{-1}$ such that we have the following properties:
$H_{\epsilon^{-1}}$ induces a map $\Uhat \to \widehat{\Delta^2_{\epsilon^{-1}}/\Gamma}$ such that the
image of the K\"ahler structure $(\Imod_{\zeta(t)}, \gmod_{\zeta(t)})$ is
$(\Imod_{\epsilon^{-2}\zeta(t)}, \epsilon^2\gmod_{\epsilon^{-2}\zeta(t)})$.  Up to rescaling the metric
by a factor $\epsilon^{-2}$, the family of K\"ahler structure converges to to the
tangent graviton graviton $Y_{\dot\zeta}$ which is smooth (not in Kronheimer's
wall).  Thus by \ref{sec:krogra} (ii), we have uniform estimates in $t$ as
$R\to +\infty$ of the form
$$
\gmod_{\epsilon^{-2}\zeta(t)} = \geuc + \xi_t \quad \mbox {with } \xi_t=\cO(R^{-4})
$$
where $\geuc$ is the Euclidean metric on $\setC^2/\Gamma$, and $R$ is the Euclidean
distance to the origin. The function $R$ is related to $r$ by the scaling
factor $r=\epsilon R$. Similar estimates hold for the complex structures
$$
\Imod_{\epsilon^{-2}\zeta(t)} = \Ieuc + \cO(R^{-4}),
$$
where $\Ieuc$ is the canonical complex structure of $\setC^2/\Gamma$.

\subsection{Background Hermitian metrics}
\label{sec:bg}
Let $\gbar$ be the orbifold K\"ahler metric on $(\Xbar,J_0)$ with K\"ahler form
$\ombar$. Recall that the isomorphism between $U$ and $\Delta^2/\Gamma$ is chosen so
that the K\"ahler form can be written as \eqref{eq:kahlloc}.  We modify
$\gbar$ so that it is flat near the the singularity. For this purpose, we
need to choose a gluing scale of the form
$$
b=\epsilon^\beta= t^{p\beta/2}
$$
where $\beta$ is a constant very close to $\frac 12$ and such that $1>\beta>\frac
12$. The precise value of this constant will be decided later on
(cf. \eqref{eq:fixbeta}).

We also need to define suitable cut-off functions.  We fix a standard bump
function $\chi:\RR\to\RR$, that is a smooth non decreasing function such that
$\chi(x)=0$ for $x\leq 0$ and $\chi(x)=1$ for $x\geq 1$. Then we choose a pair of real
parameters $(x_1,x_2)$ such that $0<x_1<x_2$ and define
\begin{equation}
  \label{eq:cutoff}
  \chi_{(x_1,x_2)}(x)=\chi\left (\frac{x-x_1}{x_2-x_1}\right)
\end{equation}

Let $r$ be the function on $U$ corresponding to the Euclidean distance
to the origin via the isomorphism $U\simeq \Delta^2/\Gamma$. 
 Then we define
\begin {align*}
\omega_{B,t} & =\ombar  \quad \mbox{
  away from the domain $r\leq 4 b$ of $U$ and} \\
\omega_{B,t} &=dd^c_{J_0}(r^2 + \chi_{(2b,4b)} (r) \eta) \quad \mbox{
  on the domain $r\leq 4b$ of $U$} 
\end{align*}
so that $\omega_{B,t}$ is the K\"ahler form of a K\"ahler orbifold metric on
$(\Xbar,J_0)$ for $t$ small enough.

Similarly, we modify the model metric $\gmod_{\epsilon^{-2}\eta(t)}$ on
$\widehat{\setC^2/\Gamma}$ as follows:
\begin{align*}
g_{A,t}&=\gmod_{\epsilon^{-2}\zeta(t)} \quad \mbox{on the domain
  $R\leq  b/\epsilon$}
\\
g_{A,t}&=\geuc \quad \mbox{on the domain $2  b/\epsilon \leq R$,}
\\
g_{A,t} &= \geuc  + (1-\chi_{( b/\epsilon ,2 b/ \epsilon )}(R)) \xi_t \quad \mbox{on the
  annulus $ b/\epsilon\leq R\leq 2 b/\epsilon$.}
\end{align*}
Hence $g_{A,t}$ defines a Riemannian metric on $\widehat {\setC^2/\Gamma}$ for $t$
small enough.

The homothety $H_{\epsilon^{-1}}$ identifies the annuli $b\leq r\leq 4 b$ and $ b/\epsilon \leq R\leq
4 b/ \epsilon$.  By construction $g_{B,t}$ is Euclidean at $ r= 2 b$ and so is
$g_{A,t}$ near $R=2 b/\epsilon$. Identifying the annuli via $H_{\epsilon^{-1}}$, we
define a Riemannian metric on $\Xhat$ by
\begin{align*}
  \tilde h_t &= \epsilon^2 H_{\epsilon^{-1}}^*g_{A,t} \quad \mbox { on the domain
  $r\leq 2 b$ of $\Uhat$} \\
 &= g_{B,t} \quad \mbox { outside the domain
  $r\leq 2 b$ of $\Uhat$}
\end{align*}

%\subsubsection*{Hermitian metrics} 
By definition, the metric $\tilde h_t$ is $J_t$-Hermitian on the
domain 
$r \leq b$ of $\Uhat$ and $J_0$-Hermitian on the complement of $r\leq 2b$. 
We construct a globally $J_t$-Hermitian metric on $\Xhat$ by
introducing its projection
\begin{equation}
  \label{eq:hversion}
h_t(u,v) = \frac 12 (\tilde h_t(u,v)+ \tilde h_t(J_tu,J_tv))  .
\end{equation}
Similarly, one can construct $\Imod_{\epsilon^{-2}\zeta(t)}$-Hermitian metrics
$h_{A,t}$ on $\widehat{\setC^2/\Gamma}$ deduced from $g_{A,t}$.

\subsection{H\"older spaces}
\label{sec:holder}
Elliptic operators, Laplacians for instance, may not be Fredholm on a
singular or noncompact
manifold. At this point, we ought to introduce suitable weighted
H\"older spaces to deal with this issue.

\subsubsection*{H\"older spaces for ALE spaces}  We
consider a \emph{radius} function $\rho_A$ on $\widehat{\setC^2/\Gamma}$. That is
a smooth function such that $\rho_A>0$ with the property that $\rho_A= R$,
say on the domain $R\geq 2$.  For $\delta\in \RR$, we define the weighted
H\"older $C^{k,\alpha}_\delta$-norm given by
$$ \|f\|_{C^{k,\alpha}_\delta(\widehat{\setC^2/\Gamma},t)}
 = \sum_{j=0}^k\sup |\rho_A^{j-\delta}\nabla^j f| + \| \rho_A^{j-\delta-\alpha}\nabla^kf \|_\alpha
$$ 
where ($i_0>0$ being fixed and chosen smaller than the injectivity radius)
$$ \| f \|_\alpha = \sup_{d(x,y)\leq i_0} \frac{ |f(x)-f(y)| }{d(x,y)^\alpha}. $$
Here, $f$ can be any tensor, and the pointwise norms are taken with
respect to the metric $\gmod_{\epsilon^{-2}\zeta(t)}$ (or alternatively $g_{A,t}$). 

\begin{rmk}
  All the norms on $\widehat{\setC^2/\Gamma}$ obtained in such a way are
 commensurate uniformly in $t$. Thus, any of these norms could be used
 in the estimating process.
\end{rmk}

\subsubsection*{H\"older spaces on orbifold}
Similarly, we define a radius function $\rho_B$ on $\Xbar$, that is a
smooth function such that $\rho_B>0$ and $ \rho_B=r$  on $U$.
We define the $C^{k,\alpha}_\delta(\Xbar)$ with the same formula as above, using
instead the metric $\gbar$ (or alternatively $g_{B,t}$).

\subsubsection*{H\"older spaces on the gluing}
Finally a third weighted norm can be defined on $\Xhat$ itself.
For this purpose, we define a weight function $\rho$ (depending on
the parameters $t$) on
$\Xhat$ as follows:
put $\rho=\rho_B$ on the complement of $\Uhat$,
  $\rho = r$ on
the domain $b\leq r\leq 1$ of $\Uhat$ and  $\rho = \epsilon H_{\epsilon^{-1}}^*\rho_A$
on the rest. 
Using the same formula as above with metrics $h_t$ (or
alternatively $\tilde h_t$), we define a norm
$\|f\|_{C^{k,\alpha}_\delta(\Xhat,t)}$ on $\Xhat$. 

Given a $m$-form $f$ on $\Xhat$, we can interpret a
$C^{k,\alpha}_\delta(t)$-estimate on $f$ as an estimate on each piece of
the manifold.
We decompose $f=f_A+f_B$ where $f_A=(1-\chi_{(b,2b)}(\rho))f$
and $f_B= \chi_{(b,2b)}(\rho) f$. Then $f_A$ is supported on the domain
$\rho\leq 2b$ and $f_A=f$ on the domain $\rho\leq b$.
Similarly $f_B$ is supported in the domain $\rho\geq b$ and $f=f_B$
on the domain $\rho\geq 2b$. 
Then  $\|f\|_{C^{k,\alpha}_\delta(\Xhat,t)}$ is uniformly commensurate
with the norms 
\begin{equation}
  \label{eq:comparenorms}
   \| f_B\|_{{C^{k,\alpha}_\delta}(\Xbar,t)}+ \epsilon^{-m-\delta}\| (H_{\epsilon^{-1}})_*f_A\|_{{C^{k,\alpha}_\delta}(\widehat{\setC^2/\Gamma},t)}
\end{equation}
where $m$ is the degree of the form $f$ and $H_{\epsilon^{-1}}$ is used to identify
the domain $\rho\leq 2b$ with the domain $\rho_A\leq 2b/\epsilon$.

\begin{rmk}
  The asymptotics of the metrics $\gbar$, $\gmod_{\epsilon^{-2}\zeta(t)}$ and complex
  structures $\Imod_{\epsilon^{-2}\zeta(t)}$, together with the naturality of the
  constructions of the 
  metrics $g_A$, $g_B$, $h_A$, $\tilde h_t$, $h_t$ and the fact that the functions
  $\chi_{(c,2c)}$ are uniformly bounded in $C^k_0$-norm imply that the
  H\" older norms defined using any of these metrics lead to uniformly
  commensurate 
  norms. Therefore, we could use any of these metrics for estimating
  the $C^{k,\alpha}_\delta$-norms. 
\end{rmk}

\subsection{Background Laplacian}
The Cauchy-Riemann operator  on $(\Xhat,J_t)$  is denoted $
\delb_{J_t}$ or $\delb_t$. Its  adjoint  deduced from the
Hermitian metric $h_t$ is denoted $\delb^*_t$. Then the Dolbeault
Laplacian is given by
$$\Box_t= \delb^*_t \delb_t+ \delb_t\delb^*_t$$ for $(p,q)$-forms
on $(\Xhat,J_t)$.

\subsubsection*{Error terms}
It should be pointed out that the Dolbeault Laplacian $\Box_t$ need not agree
with the Riemannian Laplacian $\Delta_t$ of $h_t$. Indeed, the metric $h_t$ is
not necessarily K\"ahler. Thus, $h_t$ is Hermitian, by construction, but its
K\"ahler form $\varpi_t=h_t(J_t\cdot,\cdot)$ is not a priori closed. In this section, we
investigate how close $h_t$ is to be K\"ahler.  In particular, we prove the
following lemma:
\begin{lemma}
\label{lemma:estid}
For every $k\geq 0$, there exist a constant $C_k>0$ such that for
every $\delta<0$ sufficiently close to $-2$, $t>0$ and $\beta=\frac
2{2-\delta}$, we have
$$  \| d\varpi_t\|_{C^{k,\alpha}_{\delta-1}(t)} \leq C_k \epsilon^2.
$$
Similarly, we have an estimate
$$  \| \nabla^{LC}J_t\|_{C^{k,\alpha}_{\delta-1}(t)} \leq C_k\epsilon^2.
$$
\end{lemma}
\begin{proof}
  On the compact domain $\Xhat\setminus \Uhat$, the family of complex structure
  $J_t$ converges smoothly to $J_0$, and $|J_t-J_0| =\cO(t^\ordre)$ where $\ordre$ is
  the order of the covering (\ref{ordre:revetement}). By assumption, the
  variation is non degenerate ($\ordre \geq p$), hence $|J_t-J_0|=\cO(\epsilon^2)$ on the
  compact domain. It follows that $\varpi_t$ converges smoothly to $\ombar$, the
  K\"ahler form of the orbifold metric $\gbar$. Since $d\ombar=0$, the
  estimate $\| d\varpi_t\|_{C^{1,\alpha}} =\cO(\epsilon^2)$ on this domain follows.

  On the domain $r\leq b$, the metric $\tilde h_t$ is $J_t$-Hermitian. In fact
  it agrees with the model metric $\gmod_{\epsilon^{-2}(\zeta(t)}$ up to a scaling
  factor $\epsilon^2$. Therefore $\tilde h_t=h_t$ and the $J_t$-Hermitian metric
  is K\"ahler on this domain. In particular $d\varpi_t=0$.

  On the domain $2b\leq \rho \leq 1$, we have $\tilde h_t=g_{B,t} $ and $J_t=
  \Imod_{\zeta(t)}=H_t^*\Imod_{\epsilon^{-2}\zeta(t)}$. However $\Imod_{\epsilon^{-2}\zeta(t)}=
  \Ieuc+\cO(R^{-4})$. It follows that $J_t=J_0 +\cO(\epsilon^4 r^{-4})$.  Since
  $g_{B,t}$ is K\"ahler w.r.t. the complex structure $J_0$, we have $h_t=
  g_{B,t} + \cO(\epsilon^4 r^{-4})$ and $\varpi_t = \omega_{B,t} + \cO(\epsilon^4 r^{-4})$.  Thus,
  $d\varpi_t = \cO(\epsilon^4 r^{-5})$ hence $r^{1-\delta}d\varpi_t = \cO(\epsilon^{4-\beta(5+\delta)})$ on the
  domain $2b \leq r \leq 1$. That is to say $\|d\varpi_t\|_{C^{k,\alpha}_{\delta-1}}
  =\cO(\epsilon^{4-\beta(5+\delta)})$ on the annulus $2b \leq r \leq 1$. We see that if delta is
  sufficiently close to $-2$ then the error term is a $\cO(\epsilon^2)$.

  On the domain $b \leq \rho \leq 2b$, we can rescale via the homothety $H_t$ and
  look at the construction on the annulus $ b/\epsilon\leq R\leq 2b/\epsilon$ of $\widehat
  {\setC^2/\Gamma}$.  Up to a scaling factor $\epsilon^2$, the metric $\tilde h_t$
  corresponds to $g_{A,t}$. On the other had, we have $g_{A,t} =
  \gmod_{\epsilon^{-2}\zeta(t)}+ \cO(R^{-4})$.  It follows that the
  $\Imod_{\epsilon^{-2}\zeta(t)}$-Hermitian metric $h_{A,t}$ deduced from $g_{A,t}$
  satisfies the estimate $h_{A,t}= \gmod_{\epsilon^{-2}\zeta(t)}+ \cO(R^{-4})$ and we
  have a similar estimate for the corresponding K\"ahler forms $\varpi_{A,t}$ and
  $\omod_{\epsilon^{-2}\zeta(t)}$ defined using $\Imod_{\epsilon^{-2}\zeta(t)}$.  Hence $d\varpi_{A,t}
  = \cO(R^{-5})$.  Using the homothety again, we obtain the estimate
  $d\varpi_t=\cO(\epsilon^4r^{-5})$ (there is an factor $\epsilon$ coming from the fact that
  we are taking the norm of a $3$-form for the rescaled metric) on the
  annulus $b\leq r \leq 2b$. As in the previous case, we deduce an estimate
  $\|d\varpi_t\|_{C^{k,\alpha}_{\delta-1}}=\cO(\epsilon^{4-\beta(5+\delta)})$ on the annulus $b\leq r\leq 2b$ as
  well.
\end{proof}
  From now on, we shall fix the gluing scale $b=\epsilon^\beta$, with
  the convention
  \begin{equation}
\label{eq:fixbeta}
\fbox{$\beta = \frac 2{2-\delta}$}    
  \end{equation}
as in the above lemma. The point is that for $\delta\in(-2,0)$, we
have $\beta \in (\frac 12, 1)$, $\lim_{\delta\to -2}\beta =\frac 12$.

Then we deduce that the K\"ahler form $\varpi_t$ is almost harmonic in
the sense of the following corollary:
\begin{cor}
\label{cor:estidelta}
For $\delta<0$ sufficiently close to $-2$ 
there exists a constant  $C>0$ such that for
all  $t>0$
$$  \| \Box_t\varpi_t\|_{C^{0,\alpha}_{\delta-2}(t)} \leq C\epsilon^2, $$
and 
$$  \| \Delta_t\varpi_t\|_{C^{0,\alpha}_{\delta-2}(t)} \leq C\epsilon^2. $$
\end{cor}
\begin{proof}
  Using Lemma~\ref{lemma:estid}, together with the fact that
  $*\varpi_t=\varpi_t$ we deduce an estimate
$$
 \| d^*\varpi_t\|_{C^{1,\alpha}_{\delta-1}(t)}  \leq C_1\epsilon^2.
$$
Therefore we have an estimate on $dd^*\varpi_t$ and $d^*d\varpi_t$ in
$C^{0,\alpha}_{\delta-2}(t)$-norm and the result follows for
$\Delta_t\varpi_t$.

The same proof works with the Dolbeault Laplacian. We merely use the
fact that the norm of
$*\delb_t\varpi_t$ is controlled by the norm of
$d\varpi_t$. The next step to deduce a control on
$\delb_t*\delb_t\varpi_t$. The pointwise norm of this tensor is
controlled by the pointwise norm of $\nabla^{Chern}_t
*\delb_t\varpi_t$. The metric being Hermitian, we only have
$\nabla^{Chern}_t= \nabla^{LC}_t + T$ where $T$ is a tensor such that
$T=\cO(\nabla ^{LC}_t J_t)$. Using again \ref{lemma:estid}, we
conclude that  
$$
|\nabla^{Chern}_t *\delb_t\varpi_t|\rho^{\delta-2} \leq
 |\nabla^{LC}_t *\delb_t\varpi_t|\rho^{\delta-2} + (|T|\rho) (|*\delb_t\varpi_t|\rho^{1-\delta}).
$$
The estimate follows and we have the control on $\Box_t\varpi_t$ with
the $C^{0,\alpha}_{\delta-2}$-norm.
\end{proof}

\subsubsection*{The Laplacian and gravitons}
The space $\widehat{\setC^2/\Gamma}$ is endowed with a complex structure
$\Imod_{\epsilon^{-2}\zeta(t)}$ and K\"ahler metric $\gmod_{\epsilon^{-2}\zeta(t)}$. The parameter $t=0$
corresponds to the tangent graviton.
Each of these spaces has a corresponding Laplacian $\Box^A_t=\frac 12
\Delta^A_t$. If $\delta$ is not an indicial root, the operator
$$
\Box^A_t:C^{2,\alpha}_\delta\to C^{2,\alpha}_{\delta-2}
$$
defined on $(p,q)$-forms with respect to $\gmod_{\epsilon^{-2}\zeta(t)}$ is
Fredholm. 

Indicial roots are well understood for such operators 
\begin{lemma}
\label{lemma:iroot}
Every  $\delta\in (-2,0)$ is not an indicial root.
In the case of $1$-forms every
$\delta\in (-3,1)$ is not an indicial root.
\end{lemma}
\begin{proof}
  For the first part, see \cite{RolSin05}.
For the second part, one has to check that $0$ is not an indicial
root. This boils down to check that there are no $\Gamma$-invariant
parallel $1$-forms on $\setC^2$. By duality, it follows that $-2$ is
not an indicial root either.
\end{proof}

\subsubsection*{Harmonic forms on ALE spaces}
The space of Harmonic forms of type $(p,q)$  on the ALE space, denoted $\cH^{p,q}_{A,t}$,
 is defined as the kernel of $\Box^A_{t}C^{2,\alpha}_\delta\to
 C^{2,\alpha}_{\delta-2}$, for $\delta\in(-2,0)$.
Since there are no indicial root in the interval, it follows
that the definition for  $\cH^{p,q}_{A,t}$ in independent of the
choice of $\delta \in(-2,0)$. For $1$-forms we could also choose
$\delta\in (-3,1)$ as there are no indicial roots in this interval by
Lemma~\ref{lemma:iroot}.

If we choose $\delta$ sufficiently close to $-2$, we see that harmonic
forms are in $L^2$. In fact, the decay is even better according to the
following lemma
\begin{lemma}
\label{lemma:harmdecay}
  Any harmonic form $\gamma\in \cH^{p,q}_{A,t}$ satisfies $\gamma=\cO(R^{-3})$.
\end{lemma}
\begin{proof}
It suffices to understand the case of a harmonic function $\gamma\in
C^{2,\alpha}_\delta$. The standard theory for Laplacian on ALE spaces
shows that $\gamma= cR^{-2}+\cO(R^{-3})$ since there are no indicial
roots in $(-3,-2)$. The coefficient $c$ is a
constant multiple of $\int \Box^A_t vol=0$ (cf. \cite[Theorem
8.3.6]{Joy00}) and the lemma follows. 
\end{proof}

We recall some standard results for Hodge theory on ALE spaces
(cf. \cite{Joy00} for instance):
The canonical map $\cH^{p,q}_{A,t}\to H^{p+q}(\widehat
{\setC^2/\Gamma},\setC)$ is injective with image denoted
$H^{p,q}_t$. In addition
$$
H^{k}(\widehat
{\setC^2/\Gamma},\setC) =\bigoplus_{j=0}^k H^{j,j-k}_t
$$
for all $0<k<4$. In particular we see that $\cH^{1,0}_{A,t}=\cH^{0,1}_{A,t}=0$.

We also have the following result
\begin{lemma}
  $$\cH^{2,0}_{A,t}=\cH^{0,2}_{A,t}=0$$
and
  $$\cH^{1,1}_{A,t}\simeq  H^2(\widehat{\setC^2/\Gamma,\setC})$$
for all $t$ sufficiently small.
\end{lemma}
\begin{proof}
If $t$ is sufficiently small, $\zeta'=\epsilon^{-2}\zeta(t)$ does not belong
to the wall $D$ since the variation is assumed to be non degenerate.
Similarly, let $\zeta''=((\zeta_r)'',(\zeta_c)'')\not\in D$ be a parameter
such that $\cY_{(\zeta_c)''}$ is isomorphic to 
$\widehat{\setC^2/\Gamma}$ with its canonical complex structure.
We have  $H^{2,0}(\widehat{\setC^2/\Gamma})=
H^{0,2}(\widehat{\setC^2/\Gamma})=0$ (cf. \cite[Theorem 8.4.2]{Joy00}). The semi-continuity of the
dimension of the kernel for Fredholm operator forces this property to
hold along the path from $\zeta''$ to $\zeta'$ and we deduce the lemma.
\end{proof}

\subsubsection*{Laplacian and orbifold}
One can consider the K\"ahler orbifold $(\Xbar, J_0,\gbar)$ together
with its Laplacian. Alternatively, we can look at the manifold $X$,
the smooth locus of $\Xbar$ and the Laplacian defined between weighted
H\"older spaces 
$$
\Box^B: C^{2,\alpha}_\delta\to C^{2,\alpha}_{\delta-2}
$$
acting on $(p,q)$-forms with respect to $J_0$.

Like on the ALE space, this operator is Fredholm for $\delta\in (-2,0)$,
and for $\delta \in (-3,1)$ in the case of $1$-forms.
Moreover, its kernel $\cH_B^{p,q}$ corresponds to smooth harmonic forms on $\Xbar$.

\subsection{Approximate kernel}
One can construct an approximate kernel of the operator $\Box_t$ on
$(\Xhat,J_t,h_t)$ as follows. 
The spaces $\cH^{1,1}_{A,t}$ are the fibers of a smooth (trivial)
vector bundle over the base $t\in[0,d)$.  Given $\gamma_{A,t}\in
\cH^{1,1}_{A,t}$ and $\gamma_B\in \cH_B^{p,q}$, we construct a form $\gamma_t'$ on
$\Xhat$ by requiring that $\gamma'_t=H_{\epsilon^{-1}}^*\gamma_{A,t}$ on the domain $\rho\leq
b$, $\gamma'_t= \gamma_B$ on the domain $\rho\geq 4b$ and
$$
\gamma'_{t}=(1-\chi_{(b,2b)}(r)) H_{\epsilon^{-1}}^*\gamma_{A,t} + \chi_{(2b,4b)}(r)\gamma_B
$$
Then we call $\gamma_t$ the projection of $\gamma'_t$ onto forms of
type $(1,1)$ for the complex structure $J_t$.
Thus we have constructed a linear map
\begin{equation}
  \label{eq:linmap}
\Phi_t:\cH^{1,1}_{A,t}\oplus \cH^{1,1}_{B}\to \Omega^{1,1}_{J_t}(\Xhat)  
\end{equation}
For every $t>0$ small enough,  the linear map \eqref{eq:linmap} is
injective and its image will be denoted $\cK^{1,1}_t$.

Alternatively, using the isomorphisms
$H^2(\widehat{\setC^2/\Gamma},\setC)\simeq \cH^{1,1}_{A,t}$ and
$H^{1,1}(\cX)\simeq \cH^{1,1}_{B}$, the map $\Phi_t$ can be
though of as an isomorphism
$$
\Psi_t:H^2(\widehat{\setC^2/\Gamma},\setC)\oplus H^{1,1}(\cX)\to \cK^{1,1}_t.
$$

In the case of $1$-forms, we have $\cH^{1,0}_{A,t}=0$ and
$\cH^{0,1}_{A,t}=0$.  A similar construction gives linear maps 
$$
\Phi_t: \cH^{0,1}_{B}\to \Omega^{0,1}_{J_t}(\Xhat) ,\quad
 \Phi_t:\cH^{1,0}_{B}\to \Omega^{1,0}_{J_t}(\Xhat)  
$$
These maps are injective for $t$ small enough and their images are
denoted $\cK^{0,1}_t$ and $\cK^{1,0}_t$. Then, we define isomorphisms $\Psi_t: H^{0,1}(\cX)\to \cK_t^{0,1}$, $\Psi_t: H^{1,0}(\cX)\to \cK_t^{1,0}$.

Let $\|\cdot\|_A$ be an arbitrary norm on the cohomology
$H^\bullet(\widehat{\setC^2/\Gamma},\setC)$ and $\|\cdot\|_B$ an arbitrary norm on the
cohomology $H^\bullet(\Xbar,\setC)$. We introduce a family of norms $\|\cdot\|_{\delta,t}$ on $H^2(\widehat
{\setC^2/\Gamma},\setC)\oplus H^2(\Xbar,\setC)\simeq H^2(\Xhat,\setC)$ given by
$$
\|\Omega_{A}\oplus\Omega_B\|_{\delta,t}=
\epsilon^{-2-\delta}\|\Omega_{A}\|+\|\Omega_B\| 
$$

\begin{lemma}
\label{lemma:estipsi}
Given $k\geq 0$, there are constants $c_1,c_2>0$ such that for every
$\delta \in (-2,0)$ sufficiently close to $-2$
and every $t>0$ sufficiently small, we have for all $\Omega\in H^2(\widehat
{\setC^2/\Gamma},\setC)\oplus H^{1,1}(\cX)$
$$
c_1\|\Omega\|_{\delta,t}\leq \|\Psi_t(\Omega)\|_{C^{k,\alpha}_\delta(\Xhat,t)}\leq c_2 \|\Omega\|_{\delta,t}.
$$ 
For cohomology classes $\Xi\in H^{0,1}(\cX)$, we have a similar
result with the estimate
$$
c_1\|\Xi\|_B\leq \|\Psi_t(\Xi)\|_{C^{k,\alpha}_{\delta+1}(\Xhat,t)}\leq c_2\|\Xi\|_B
$$

\end{lemma}
\begin{proof}
The injection $\cH^{1,1}_B\to \cK^{1,1}_t$  induced by $\Phi_t$ allows
to pull-back the $C^{k,\alpha}_{\delta}(\Xhat,t)$ norm. It is readily
checked that this norm is uniformly commensurate with the
$C^{k,\alpha}_\delta(\Xbar)$-norm. 

Similarly, the injection $\cH^{1,1}_{A,t}\to \cK^{1,1}_t$  induced by $\Phi_t$ allows
to pull-back the $C^{k,\alpha}_{\delta}(\widehat{\setC^2/\Gamma},t)$
norm. 
This norm is uniformly commensurate with the
$C^{k,\alpha}_\delta(\Xbar)$-norm up to a factor
$\epsilon^{-\delta-2}$ by \eqref{eq:comparenorms}. 

The proof of the second part of the statement goes along the same
lines, except that we do not have to deal with harmonic forms on the
ALE space.
\end{proof}

\subsubsection*{Uniform elliptic estimates}
A key step in order to construct K\"ahler forms is Hodge theory. More
precisely, we should control the first eigenvalues of the Dolbeault
Laplacian between weighted H\"older spaces:
\begin{prop}
\label{prop:linearapprox}
There exists a constant $c>0$ such that for all $t>0$ sufficiently
small and every $(1,1)$-form $\gamma$ on $(\Xhat,J_t)$, we have
$$
c\|\gamma\|_ {C^{2,\alpha}_{\delta}(t)}\leq
\|\gamma\|_{C^{0,\alpha}_{\delta}(t)} + \|\Box_t
\gamma\|_{C^{0,\alpha}_{\delta-2}(t)} .
$$

If in addition $\gamma$ is $L^2$-orthogonal to the space $\cK^{1,1}_t$ and
$\delta\in (-2,0)$ sufficiently close to $-2$, we have
$$
c\|\gamma\|_ {C^{2,\alpha}_{\delta}(t)}\leq
 \|\Box_t
\gamma\|_{C^{0,\alpha}_{\delta-2}(t)} .
$$
Similarly, if $\gamma$ is a $(0,1)$-form orthogonal to $\cK^{0,1}_t$, we have an estimate
$$
c\|\gamma\|_ {C^{2,\alpha}_{\delta+1}(t)}\leq
 \|\Box_t
\gamma\|_{C^{0,\alpha}_{\delta-1}(t)} .
$$
\end{prop}
\begin{proof}
  The first part of the proposition is standard.
For the second part, let us argue by contradiction. If the proposition
is not true, there are weights $\delta\in(-2,0)$ arbitrarily close to $-2$ with and  families of parameter $t_j\to 0$ and
$(1,1)$-forms  $\gamma_j$ on $(\Xhat,J_t)$ such that
\begin{equation}
  \label{eq:contrad}
\|\gamma_j\|_{C^{0,\alpha}_\delta(t_j)}=1,\mbox { and }  \|\Box_{t_j}
\gamma_j\|_{C^{0,\alpha}_{\delta-2}(t_j)} \to 0.  
\end{equation}
The first part of the proposition provides a uniform
$C^{2,\alpha}_\delta$ estimate on $\gamma_j$. 

Then, up to extraction of a subsequence, we may assume that $\gamma_j$
converges in the $C^{0,\alpha}$ sense on every compact set of $X$
toward a form $\gamma$ on $X$ which is $\gbar$-harmonic and in
$C^{2,\alpha}_\delta(\Xbar,t_j)$. This implies that $\gamma$ extends as a smooth
orbifold form on the orbifold $\Xbar$.
Since each $\gamma_j$ is orthogonal to $\cK_t$ and $\delta >-3$, it follows that 
$\gamma$ is orthogonal to $\cH^{1,1}_B$. This forces $\gamma=0$. 

Similarly, we can transport $\gamma_j$ on the domain $R\leq 4b_j/\epsilon_j$ of
$\widehat{\setC^2/\Gamma}$ using the homothety $H_j=H_{\epsilon_j^{-1}}$.  Put $\mu_j= \epsilon ^{-2-\delta}
H^*_j\gamma_j$. By definition of the norm we obtain a uniform $C^{2,\alpha}$
bound on $\mu_j$. Up to extraction, we may assume that $\mu_j$ converge on
every compact set of $\widehat{\setC^2/\Gamma}$ to a harmonic
form $\mu\in C^{2,\alpha}_\delta$ on the tangent graviton.  Again, the
fact that harmonic forms must have a 
strong decay forces $\mu\in C^{2,\alpha}_{-3}$ by 
Lemma~\ref{lemma:harmdecay}. Then,  the fact that
$\gamma_j$ is orthogonal to $\cK^{1,1}t$ implies that $\mu$ is orthogonal to
$\cK^{1,1}_{A,t}$. We conclude $\mu=0$.

Let $m_j$ be a point where the function $|\rho_j^{-\delta}\gamma_j|$
is maximal equal to $1$ (cf. assumption \eqref{eq:contrad}).  Up to
extraction of a converging subsequence, we have either
\begin{enumerate}
\item $\rho_j(m_j)\leq 2b_j$
\item $\rho_j(m_j)\geq 2b_j$
\end{enumerate}
for all $j$.

Case (ii):
If $\rho_j(m_j)$ is bounded away from
zero we clearly have a contradiction. Indeed, after further
extraction, we may assume that $m_j$ converges to a point in $X$. But
we know that $\gamma_j$ converges to $0$ on every compact set of $X$
hence hence $|\rho_j^{-\delta}\gamma_j|(m_j)\to 0$ which is
impossible.

So we may assume up to extraction that $\lim \rho_j(m_j)=0$. Using
homotheties and rescaling for $\gamma_j$ again, we can extract a converging
subsequence on every 
compact set of the cone $\setC^2\setminus 0)/\Gamma$ such that the limit
is nonvanishing, harmonic and in  $C^{2,\alpha}_\delta$ on the cone. This is not
possible for $\delta$ is not an indicial root of the Laplacian.

Case (i):
If $\rho_j(m_j)/\epsilon_j$ is bounded, we have a contradiction. The
proof is the same as in
the first case, using the rescaled forms $\mu_j$ instead.

If it is not bounded, we may assume that it goes to infinity after
extraction. Then we use rescaling to extract a harmonic limit on the
cone, exactly as in the first case.
\end{proof}

\subsubsection*{From  approximate kernel to harmonic forms}
The spaces $\cK_t$ consist of forms which are approximately harmonic
in the sense of the following lemma.
\begin{lemma}
\label{lemma:estilaplpsi} Fix $-2<\delta<0$ sufficiently close to $-2$.
There is a constant $c>0$ such that for all $\Omega_A\in
H^2(\widehat{\setC^2/\Gamma},\setC)$, $\Omega_B\in H^{1,1}(\cX)$ and $\Omega=\Omega_A\oplus \Omega_B$, we have
$$
c\|\Box_t \Psi_t(\Omega_A) \|_{C^{0,\alpha}_{\delta-2}}\leq  \epsilon^{2-\beta(\delta+4)} \|\Omega_A\|_A
$$
and
$$
c\|\Box_t \Psi_t(\Omega_B) \|_{C^{0,\alpha}_{\delta-2}}\leq
\epsilon^{ -\beta\delta} \|\Omega_B\|_B
$$
In particular
$$
c\|\Box_t \Psi_t(\Omega) \|_{C^{0,\alpha}_{\delta-2}}\leq \epsilon^{\beta_1} \|\Omega\|_{\delta,t}
$$
where $\beta_1 = \min(-\beta\delta, 4+\delta -\beta(\delta+4))$ is
very close to $1$, by definition.

Similarly, if $\Xi\in H^{0,1}(\cX)$, we have
$$
c\|\Box_t \Psi_t(\Xi)\|_{C^{0,\alpha}_{\delta -1}}\leq \epsilon^{-\beta(\delta+1)}\|\Xi\|_B.
$$
\end{lemma}
\begin{proof}
  The proof is similar to the one for Lemma~\ref{lemma:estid} and
  Corollary \ref{cor:estidelta}. 
Let $\gamma_{A,t}\in \cH^{1,1}_{A,t}$ be a family of harmonic
$(1,1)$-forms on the family of ALE representing $\Omega_A$.
Then, we have uniform estimates $\gamma_{A,t}= \cO(R^{-4})$ on the
annulus $b/\epsilon \leq R\leq 4b/\epsilon$. Using the rescaled
metric, we deduce an estimate
$|\Psi_t(\Omega_A)|=\cO(\epsilon^2r^{-4})$ on the annulus $b\leq r\leq
4b$.
Hence $|\Box_t\Psi_t(\Omega_A)|=\cO(\epsilon^2r^{-6})$ so
$|r^{2-\delta}\Box_t\Psi_t(\Omega_A)|=\cO(\epsilon^2r^{-4-\delta})$,
that is to say
$|r^{2-\delta}\Box_t\Psi_t(\Omega_A)|=\cO(\epsilon^{2-\beta(4+\delta)})$
on the annulus. The first part of the lemma follows.

For the second part of the statement,
we start with the uniform estimate $\Psi_t(\Omega_B)=\cO(1)$ on the
annulus 
$b\leq r\leq 4b$. It follows that
$|r^{2-\delta}\Box_t\Psi_t{\Omega_B}|=\cO(\epsilon^{-\beta\delta})$ on
the annulus. 

On
the annulus $4b\leq r\leq 1$ we have the estimate
$|J_0-J_t|=\cO(R^{-4})=\cO(\epsilon^4r^{-4})$. It follows that
$|\Box_t\Psi_t(\Omega_B)|= \cO(\epsilon^4r^{-6})$ on this annulus.
Thus
$|r^{2-\delta}\Box_t\Psi_t(\Omega_B)|=\cO(\epsilon^{4-\beta(4+\delta)})$
on the annulus $4b\leq r\leq 1$. We see that $4-\beta(4+\delta)$ goes
to $3$ as $\beta$ is close to $-2$ and $\beta$ close to $1/2$.

On the compact part, we have an estimate  $J_0-J_t=\cO(\epsilon^2)$.
So the estimate $|\Box_t\Psi_t(\Omega_B)|=\cO(\epsilon^2)$ on the
compact domain follows. 

The second part of the lemma follows.
The third inequality is obvious.

For the last statement, we just notice that the same estimates hold in
the case of $1$-forms. We merely have to replace $\delta$ by
$\delta+1)$. So we have the estimate $\cO(\epsilon
^{-\beta(\delta+1)})$ on the annulus $b\leq r\leq 4b$, the estimate 
$\cO(\epsilon ^{4-\beta(5+\delta)})$ on the annulus $4b\leq r\leq 1$
and $\cO(\epsilon^2)$ on the compact domain.
\end{proof}

Let $P_t: \cH^{1,1}(\Xhat,J_t)\to \cK^{1,1}_t$ be the $L^2$-orthogonal
projection (deduced from $h_t$) on the space $\cK^{1,1}_t$. Similarly,
we denote $P_t: \cH^{0,1}(\Xhat,J_t)\to \cK^{0,1}_t$. This
projection is very close to the identity in the sense of the following
corollary.
\begin{cor}
  Suppose $-2<\delta<0$ with $\delta$ sufficiently close to $-2$.
  There exists a constant $c>0$ such that for all 
  $\gamma_t\in \cH^{1,1}(\Xhat,J_t)$
$$
\|\gamma_t - P_t(\gamma_t)\|_{C^{2,\alpha}_\delta}\leq c\epsilon^{\beta_1} \|P_t(\gamma_t)\|_{C^{2,\alpha}_\delta}.  
$$
Similarly, if $\gamma_t$ denote a family $\gamma_t\in
\cH^{0,1}(\Xhat,J_t)$ we have
$$
\|\gamma_t - P_t(\gamma_t)\|_{C^{2,\alpha}_{\delta+1}}\leq c\epsilon^{-\beta(\delta+1))} \|P_t(\gamma_t)\|_{C^{2,\alpha}_{\delta+1}}.  
$$
\end{cor}
\begin{proof}
  We write $P_t(\gamma_t)=\Psi_t(\Omega_t)$ and consider the form
  $\eta_t=P_t(\gamma_t)-\gamma_t$. By definition $\eta_t$ is orthogonal to
  $\cK_t$ and $\Box_t\eta_t=\Box_tP_t(\gamma_t)=\Box_t\Psi_t(\Omega_t)$.  
Lemma~\ref{lemma:estilaplpsi} applied to $\Psi_t(\Omega_t)$ followed by 
Lemma~\ref{lemma:estipsi} and Proposition~\ref{prop:linearapprox} give the
  result.
\end{proof}

In conclusion, $P_t$ is an isomorphism and the operator norm
$\|P_t-\id\|_{C^{2,\alpha}_\delta}$ is $\cO(\epsilon^{\beta_1})$ (and
$\cO(\epsilon^{-\beta(\delta+1)})$ in the case of $1$-forms) and the proposition
below follows.
\begin{prop}
\label{prop:linear}
For all $k\geq 0$ there exists a constant $c>0$ such that for all $\delta\in
(-2,0)$ sufficiently close to $-2$, $t>0$ and every $(1,1)$-form $\beta$ orthogonal to harmonic forms
on $(\Xhat,J_t,h_t)$ , and we have
$$
c\|\beta\|_ {C^{2+k,\alpha}_{\delta}(t)}\leq
 \|\Box_t
\beta\|_{C^{k,\alpha}_{\delta-2}(t)} .
$$
In the case of $(0,1)$-forms orthogonal to harmonic forms, we have a similar estimate with
$$
c\|\beta\|_ {C^{2+k,\alpha}_{\delta+1}(t)}\leq
 \|\Box_t
\beta\|_{C^{k,\alpha}_{\delta-1}(t)} .
$$
\end{prop}

\subsection{Background K\"ahler structure}
Recall that $(J_t,h_t)$  is a Hermitian structure on $\Xhat$. However
$h_t$ is not a priori K\"ahler. We shall look for a nearby Hermitian
metric which is K\"ahler.

For $t>0$, the K\"ahler form $\varpi_t$ of $h_t$ admits a decomposition of the
form
$$
\varpi_t= \varpi^H_t + \varpi^\perp_t
$$
where $\varpi^H_t$ is a $\Box_t$-harmonic $(1,1)$-form on
$(\Xhat,J_t)$
and $\varpi^\perp_t$ is 
$L^2$-orthogonal to $\Box_t$-harmonic forms.

From this point, one can prove that the following proposition
\begin{prop}
\label{prop:hermharm}
Assuming that $\delta\in(-2,0)$ and sufficiently close to $-2$, we have
$\|\varpi_t- \varpi^H_t \|_{C^{2,\alpha}_{\delta}(\Xhat,t)}= \cO(\epsilon^2)$.
\end{prop}
\begin{proof}
  By definition, $\varpi_t- \varpi^H_t=\varpi^\perp_t$ and 
$\Box_t \varpi_t^\perp = \Box_t \varpi_t$.
by
  Proposition~\ref{prop:linear}, we have 
$$\|\varpi_t- \varpi^H_t
\|_{C^{2,\alpha}_{\delta}(\Xhat,t)}=\cO(\|\Box_t
\varpi_t\|_{C^{0,\alpha}_{\delta-2}(\Xhat,t)}) .
$$
The proposition follows from Corollary~\ref{cor:estidelta}.
\end{proof}

By definition $\delb_t\varpi_t^H=0$ but it is not a priori $d$-closed.  We
remedy to this problem with the following lemma
\begin{lemma}
\label{lemma:closedgamma}
  There exists a $(1,1)$-form $\gamma_t$ on $(\Xhat,J_t)$ in the cohomology
  class of $\varpi^H_t$ such that $d\gamma_t=0$ and $\|\gamma_t- \varpi^H_t
  \|_{C^{2,\alpha}_{\delta}(t)}= \cO(\epsilon^2)$.
\end{lemma}
\begin{proof}
  Since $(\Xhat,J_t)$ is K\"ahler, the Fr\"olicher exact sequence
  degenerates at the first page. In particular $\del_{t}\varpi^H_t =
  \delb_{t}\alpha_t$, for $\alpha_t$ a $(2,0)$-form w.r.t $J_t$. Here we may
  choose a form such that $\delb^*_{t}\alpha_t=0$ and $\alpha_t$ is orthogonal
  to $\Box_t$-harmonic forms.

  So $\delb_t\bar\alpha_t=0$ because we are working in complex dimension
  $2$. So $\bar\alpha_t$ defines a class in $H^{0,2}(\Xhat,J_t)$ and we can
  write $\bar \alpha_t = \bar\mu_t + \delb_{t}\beta_t$, where $\mu_t$ is a
  holomorphic $(2,0)$-form, $\beta_t$ is a $(0,1)$-form orthogonal to
  harmonic forms which satisfies $\delb_t^*\beta_t=0$.  Put
  $\gamma_t=\varpi^H_t+\delb_{t}\bar\beta_t$. Then $\delb_{t}\gamma_t=\delb_t\varpi^H_t=0$ and
  $\del_{t}\gamma_t =\del_{t}\varpi^H_t + \del_{t}\delb_{t}\bar\beta_t=0 $.

  The next step is to estimate $\beta_t$. Using the identities
  $\delb_t^*\del_t\varpi_t^H = \delb_t^*\delb_t\alpha_t$ and $\delb_t^*\alpha_t=0$ we
  obtain $\delb_t^*\del_t\varpi_t^H=\Box_t \alpha_t$.

  Now
$$\|\delb_t^*\del_t(\varpi_t^H - \varpi_t)\|_{C^{0,\alpha}_{\delta-2}}=\cO(\epsilon^2)
$$
according to Proposition~\ref{prop:hermharm}.  Therefore
$$\|\delb_t^*\del_t\varpi_t^H \|_{C^{0,\alpha}_{\delta-2}}\leq
\|\delb_t^*\del_t\varpi_t^H \|_{C^{0,\alpha}_{\delta-2}} + \cO(\epsilon^2).
$$
Using Lemma~\ref{lemma:estid}
we obtain the estimate $\|\delb_t^*\del_t\varpi_t^H
\|_{C^{0,\alpha}_{\delta-2}}=\cO(\epsilon^2)$ and we deduce 
$$ \|\Box_t\alpha_t \|_{C^{0,\alpha}_{\delta-2}}=\cO(\epsilon^2). $$
Proposition \ref{prop:linear} provides the estimate 
\begin{equation}
  \label{eq:estialpha}
 \|\alpha_t\|_{C^{2,\alpha}_{\delta}}=\cO(\epsilon^2).
\end{equation}

Eventually, we would like to estimate the $(0,1)$-forms $\beta_t$.
We have $\delb^*_t\alpha_t = \delb^*_t\delb_t\beta_t=\Box_t\beta_t$
using the fact that $\bar\mu$ is harmonic and $\delb^*_t\beta_t=0$.

The estimate \eqref{eq:estialpha} provides an estimate
$\|\delb_t^*\alpha_t\|_{C^{1,\alpha}_{\delta-1}}=\cO(\epsilon^2)$
and it follows that 
$$\|\Box_t\beta_t\|_{C^{1,\alpha}_{\delta-1}}=\cO(\epsilon^2).$$
Using the fact that $\beta_t$ is orthogonal to harmonic forms and
Proposition~\ref{prop:linear}, we get the estimate
$$\|\beta_t\|_{C^{3,\alpha}_{\delta+1}}=\cO(\epsilon^2).$$
In particular this implies
$$\|\delb_t\bar\beta_t\|_{C^{2,\alpha}_{\delta}}=\cO(\epsilon^2)$$
which proves the lemma.
\end{proof}

Thus we define the real closed form $\omega'_t$ of type $(1,1)$ on $(\Xhat,J_t)$
by taking
$$
\omega'_t=\Re(\gamma_t)
$$
were $\gamma'_t$ is given by Lemma~\ref{lemma:closedgamma}.

\begin{cor}
\label{cor:bgmet}
Fix $-2<\delta<0$ sufficiently close to $-2$.  Define $\omega'_t$ to be
the real part of $\gamma_t$. Then for all 
$t>0$ sufficiently small, $\omega'_t$ defines a K\"ahler metric $g'_t$ on
$(\Xhat, J_t)$ such that $\|\varpi_t-\omega'_t\|_{C^{2,\alpha}_\delta(t)}=\cO(\epsilon^2)$.
\end{cor}
\begin{proof}
By definition we have
$\|\varpi_t-\omega'_t\|_{C^{2,\alpha}_\delta(t)}=\cO(\epsilon^2)$.
So we only need to check that the estimate is good enough to ensure
that the real $(1,1)$-form $\omega'_t$ defines a metric.

On the domain $\rho\leq 4b$ of $\Uhat$, using the homothety
$H_{\epsilon^{-1}}$, we find an estimate
$$
\|\omod_{\epsilon^{-2}\eta(t)}-\epsilon^{-2}(H_{\epsilon^{-1}})_*\omega'_t\|_{C^{2,\alpha}_\delta(\rho_A\leq
  4\epsilon^{-\frac 12},t)}=\cO(\epsilon^{2+\delta})$$
Since $2+\delta>0$ the form $\omega'_t$ is definite positive for $t$
sufficiently small.
A similar estimate on the domains $4\epsilon^{\frac 12}\leq \rho\leq 1$ and
$\Xhat\setminus \Uhat$ proves the lemma.
\end{proof}

Although they need not agree, the K\"ahler class $[\omega'_t]$ is very close to $\Omega_t$ according
to the following result:
\begin{lemma}
\label{lemma:estikahlclass}
  The cohomology class $[\omega'_t]\in H^2(\Xhat,\RR)$ satisfies 
$$
\|\Omega_t-[\omega'_t]\|_{\delta,t}=\cO(\epsilon^2).
$$
\end{lemma}
\begin{proof}
 We have the decompositions $\Omega_t=\Omega_{t,A}+\Omega_{t,B}$
  and $[\omega'_t]=[\omega'_t]_A+[\omega'_t]_B$ together with the estimate
  $\|\omega'_t-\varpi_t\|_{C^{2,\alpha}_\delta(t)}=\cO(\epsilon^2)$. On the ALE part, the form $\varpi_t$
  is closed and represents $\Omega_{A,t}$. From $|\omega'_t-\varpi_t|=\cO(\epsilon^{2+\delta})$
  and the fact that the spheres representing the homology classes have
  $\varpi_t$-volume $\cO(\epsilon^2)$, we deduce that
  $$ |[\omega'_t]_A-\Omega_{t,A}| = \cO(\epsilon^{4+\delta}). $$
  On the compact part, we have $|\omega'_t-\varpi_t| =\cO(\epsilon^2)$
  and by assumption
  $|\Omega_{t,B}-\Omega_0| =\cO(\epsilon^2)$. Therefore
  $$ |[\omega'_t]_B-\Omega_{t,B}| = \cO(\epsilon^2). $$
  These two estimates together prove the lemma.
\end{proof}

\subsection{Harmonic forms and K\"ahler form}
We shall now use the family of K\"ahler metrics $g'_t$ with K\"ahler form
$\omega'_t$ on $(\Xhat,J_t)$ provided by Corollary \ref{cor:bgmet} as our
favorite background metric for some $\delta\in (-2,0)$ sufficiently
close to $-2$. Its Laplacians will
be denoted as well $\Box_t$ and $\Delta_t=2\Box_t$.

The Mayer-Vietoris isomorphism $H^2(\Xhat)\to H^2(\Xbar)\oplus
H^2(\widehat{\setC^2/\Gamma})$, gives a corresponding decomposition $\Xi_A\oplus\Xi_B$
for each cohomology class $\Xi\in H^2(\Xhat)$.  The family of norms $\|\Xi_A\oplus
\Xi_B\|_{\delta,t}$ defined at \S\ref{sec:holder} provides a norm on $H^2(\Xhat)$ denoted
$\|\Xi\|_{\delta,t}$ as well.

Then we have the following result
\begin{prop}
\label{prop:equivnorm}
There are constants $c_1,c_2>0$ such that for every family $\gamma_t$ of
$g'_t$-harmonic $2$-forms on $\Xhat$ with cohomology class $\Xi_t$, we
have
$$
c_1 \|\gamma_t\|_{C^{0,\alpha}_{\delta}(t)} \geq \|\Xi_t\|_{t,\delta}\geq c_2 \|\gamma_t\|_{C^{0,\alpha}_{\delta}(t)}.
$$
\end{prop}
\begin{proof}
  The proof is done by contradiction. If the second inequality does
  not hold,
  there exists a family $t_j\to 0$ and $g'_{t_j}$-harmonic $2$-forms
  $\gamma_{t_j}$  such 
  that  $\|\Xi_{t_j}\|_{\delta,t_j}\to 0$ and
  $\|\gamma_{t_j}\|_{C^{0,\alpha}_{\delta}(t_j)}=1$.

  Elliptic regularity gives a uniforms $C^{2,\alpha}_\delta$ estimate on
  $\gamma_{t_j}$ by Proposition~\ref{prop:linearapprox}. Arguing  exactly as in the proof of of the
  second part of  Proposition~\ref{prop:linearapprox} by 
  extracting converging subsequences, we show that either
  \begin{enumerate}
  \item   $\Xi_{B,t_j}$ 
  converges to a non vanishing cohomology class in
  $H^2(\Xhat,\RR)$
\item or $\epsilon_j^{-2-\delta}\Xi_{A,t_i}$   converges to a non vanishing cohomology class in
  $H^2(\Xhat,\RR)$.
  \end{enumerate}
This contradicts the fact that $\lim\|\Xi_{t_j}\|_{\delta,t_j}=0$.

For the first part of the inequality, a similar proof gives the result.
\end{proof}

We do not control exactly the K\"ahler class $[\omega'_t]$ of the K\"ahler
metric $g'_t$ that was just constructed. We will construct a
nearby K\"ahler
metric $g_t$ with the K\"ahler class $\Omega_t$ by perturbing
$g'_t$.

\begin{prop}
\label{prop:omegap}
For every $t>0$ sufficiently small,
the $g'_t$-harmonic representative $\omega_t$ of the K\"ahler class $\Omega_t$
defines a K\"ahler metric $g_t$, satisfying the estimate
$$ \| \omega_t-\omega'_t \|_{C^{2,\alpha}_\delta(t)} = \cO(\epsilon^2). $$
\end{prop}
\begin{proof}
  According to Lemma~\ref{lemma:estikahlclass}
  \begin{equation}
   \| \Omega_t - [\omega_t] \|_{\delta,t} = \cO(\epsilon^2).\label{eq:7}
  \end{equation}
Then by Proposition~\ref{prop:equivnorm}
  one has $\|\omega'_t-\omega_t\|_{C^{2,\alpha}_\delta(t)}=\cO(\epsilon^2)$. The worst value of the
  weight is $\epsilon^{-\delta}$, so it implies $\epsilon^{-\delta}|\omega'_t-\omega_t|=\cO(\epsilon^2)$, that
  is
  $$ |\omega'_t-\omega_t|=\cO(\epsilon^{2+\delta}). $$
  Since $-2<\delta<0$, this goes to zero and $\omega'_t$ is a K\"ahler form for
  $t$ small enough. The proposition follows.
\end{proof}

We summarize the results of the current section in the following
theorem
\begin{theo}
  \label{theo:kahlmet}
  Let $\cX\hkto\cM\to \Delta$ be a family of deformations of a compact complex
  surfaces with canonical singularities.  Let $\ombar$ be an orbifold
  K\"ahler metric on $\cX$ with K\"ahler class $\Omega_0$ and $\Omega$ a family of K\"ahler
  classes supported by a simultaneous resolution $\cXhat\to\cMhat\to \Delta_d$
  degenerating toward $\Omega_0$, such that the variation of K\"ahler class and
  complex structure non degenerate.

  Let $\phi:\Delta_d\times \Xhat\to \cMhat$ be a smooth trivialization of the family and
  $h_t$ the family of Hermitian metrics on $(\Xhat,J_t)$ with K\"ahler form
  $\varpi_t$ constructed at \S\ref{sec:bg} for all $t\in\Delta_d\cap \RR^+$ sufficiently
  small.

Then,  there
exists a family of
K\"ahler metrics $g_t$ with K\"ahler form $\omega_t$ on $(\Xhat,J_t)$
defined for $t\in \Delta_d\cap \RR^+$ for all $t$ sufficiently small 
with
the property that
\begin{itemize}
\item $[\omega_t]=\Omega_t$
\item for all $\delta\in (-2,0)$ sufficiently close to $-2$, we have $\|\omega_t-\varpi_t\|_{C^{2,\alpha}_{\delta}}= \cO(\epsilon^2)$.
\end{itemize}
\end{theo}

\begin{proof}[Proof of Proposition~\ref{prop:kahlmet}]
Since $h_t$ converges smoothly toward $\gbar$ on every compact set of
$X$, the proposition is an immediate corollary of Theorem~\ref{theo:kahlmet}.
\end{proof}

\section{CSCK metrics}\label{sec:csck-metrics}
Let $\cX\hkto \cM\to\Delta$ be a flat deformation and $\cXhat\hkto\cMhat\to
\Delta_d$ a simultaneous resolution after passing to a ramified
cover. Suppose that $\cX$ is endowed with a CSCK orbifold metric
$\gbar$ with K\"ahler class $\Omega_0$. Let $\Omega$ be a family of K\"ahler classes
supported by the simultaneous resolution degenerating toward $\Omega_0$, with
non degenerate variation at each singularity.

Let $g_t$ be the family of K\"ahler metrics on $\cM_t$ with K\"ahler
class $\Omega_t$ obtained in \S~\ref{sec:rep}. In this section, we show that
$g_t$ can be perturbed into a CSCK metric. This kind of result is now
well-known by the work of Arezzo-Pacard, see also Sz\'ekelyhidi whose
proof is closer to ours. Our setting is slightly different, because
we vary the complex structure as well. Therefore we shall give quickly
some steps of the proof, but omit most technical proofs since they are
similar to that in the literature.

\subsection{Scalar curvature estimates}
We begin by the following estimate on the scalar curvature of the
K\"ahler metric $g'_t$.
\begin{prop}
\label{prop:estiscal}
Let $\kappa$ be the (constant) scalar curvature of the orbifold K\"ahler
metric on $\Xbar$. For all $\delta\in (-2,0)$, the K\"ahler metric
$g'_t$ with K\"ahler class $\Omega_t$ 
satisfies the estimate
$$
\|\scal(g'_t) -\kappa\|_{C^{0,\alpha}_{\delta-2}(t)}=  \cO(\epsilon^{2}).
$$
\end{prop}

\begin{proof}
As an immediate consequence of Proposition~\ref{prop:omegap} and
Corollary \ref{cor:bgmet}, we have the estimate
$$
\|\scal(g'_t) -\scal (h_t)\|_{C^{0,\alpha}_{\delta-2}(t)}=  \cO(\epsilon^2).
$$
Then the proposition will be the consequence of the estimate
\begin{equation}
\|\scal(h_t) -\kappa \|_{C^{0,\alpha}_{\delta-2}(t)}
  =  \cO(\epsilon^{2}),\label{eq:1}
\end{equation}
that we now prove.

By construction $\scal(g_B)=\kappa$ on the domain $\rho_B\geq 4b$ and
$\scal(g_B)= \cO(1)$ on the annulus $2b\leq r \leq 4b$.  On the other hand
$\scal(g_{A,t})=0$ on the domain $R\leq b/\epsilon$ and
$\scal(g_{A,t})=\cO(\epsilon ^{6}b^{-6})$
on the annulus $b/\epsilon\leq R\leq 2b/\epsilon$.  Hence the metric
$\tilde h_t$ obtained 
by gluing together $g_A$ and $g_B$ satisfies
$|\scal(\tilde h_t)|=\cO(\epsilon^{4}b^{-6})+\cO(1)$ on the annulus $b\leq r\leq
2b$. Therefore
$r^{2-\delta}|\scal(\tilde h_t)|=\cO(\epsilon^{4}b^{4-\delta})+\cO(b^{2-\delta})=
\cO(\epsilon^{4-\beta(2+\delta)})+ \cO(\epsilon^{\beta(2-\delta)})$ on
the annulus.  By definition $\beta\delta=2$, hence
$$
\rho^{2-\delta}|\scal(\tilde h_t)-\kappa|= \cO(\epsilon^{4-\beta(2+\delta)})+ \cO(\epsilon^2)
$$
on the annulus $b\leq\rho\leq 4b$ which give an estimate
$$
\|\scal(\tilde h_t)-\kappa\|_{C^{0,\alpha}_{\delta-2}(t)}=
\cO(\epsilon^2)  
$$
on $\Xhat$ for $\delta$ sufficiently close to $-2$.

The metric $h_t$ is obtained by projecting $\tilde h_t$ onto its
$J_t$-invariant component. The estimate \eqref{eq:1} follows from the estimates on
$J_t-J_0$ with a proof along the same lines of  Lemma~\ref{lemma:estilaplpsi}.
\end{proof}

\subsection{Construction of the metrics}
It is convenient to work with the family of complex structure $J_t$ on
a fixed smooth manifold $\Xhat$ as in \S~\ref{sec:rep}.  Then, we
consider
$$
\omega_{t,\phi}=\omega_t + dd^c_{J_t}\phi.
$$
where $\phi$ is a function on $\Xhat$.
 If $\phi$ is small enough, $\omega_{t,\phi}$
is also the K\"ahler form of a K\"ahler metric $g_{t,\phi}$ on
$(\Xhat,J_t)$ representing $\Omega_t$.
More specifically, we have the following result:
\begin{lemma}
  Let $C$ be a positive constant and $\delta\in(-2,0)$ sufficiently
  close to $-2$. Then for every $t>0$ sufficiently
  small  and every function $\phi$  such that 
$\|\phi\|_{C^{4,\alpha}_{\delta+2}}\leq
C\epsilon^{2-\beta(\delta+2)}$, the form $\omega_{t,\phi}$ is definite
positive.
\end{lemma}
\begin{proof}
We deduce an estimate 
$\|\omega_t-\omega_{t,\phi}\|_{C^{2,\alpha}_\delta}=\cO(\epsilon^{2-\beta(\delta+2)})$. The
worst value of the weight is $\epsilon^{-\delta}$ so we have an
estimate
$\epsilon^{-\delta}|\omega_t-\omega_{t,\phi}|=\cO(\epsilon^{2-\beta(\delta+2)})$,
hence
$|\omega_t-\omega_{t,\phi}|=\cO(\epsilon^{(1-\beta)(\delta+2)})$. Since
$(1-\beta)(\delta+2)>0$, $\omega_{\phi,t}$ is definite positive for
$t$ sufficiently small.  
\end{proof}

 We want to solve the equation
\begin{equation}
  \label{eq:csck}
\scal(g_{t,\phi}) = cst  
\end{equation}
where $\scal(g_{t,\phi})$ is the scalar curvature of the metric $g_{t,\phi}$.
The linearization of this equation at $\phi=0$ is given by a fourth order
elliptic operator $L_t$, the Lichnerowicz operator. The idea is to apply a
suitable version of the implicit function theorem in order to solve
\eqref{eq:csck}.
\begin{prop}
\label{prop:ri}
  Suppose that $\cX$ does not carry any nontrivial holomorphic vector
  field. If $-2<\delta<0$, then for sufficiently small $t>0$ the operator
  \begin{align*}
    P_t: \RR \times C^{4,\alpha}_{\delta+2}(\Xhat,t) &\to
    C^{0,\alpha}_{\delta-2}(\Xhat,t) \\
(v,\phi)&\mapsto v+ L_t\phi 
  \end{align*}
admits a right inverse $Q_t$ with norm satisfying
$\|Q_t\|\leq c\epsilon^{-\beta(\delta+2)}$ for some constant $c$ independent of $t$.
\end{prop}
From that and the initial control on the scalar curvature, it follows:
\begin{cor}\label{cor:CSCK}
  Suppose that $\cX$ does not carry any nontrivial holomorphic vector
  field. For all $\delta\in(-2,0)$ sufficiently close to $-2$, there exists a
  constant $C>0$ such that for all $t>0$ sufficiently small, there is a
  unique solution $\phi_t$, up to a constant, to the equation $\scal(g_{t,\phi})
  = cst$ with the condition that $\|\phi_t\|_{C^{4,\alpha}_{\delta+2}}\leq C\varepsilon^{2-\beta(\delta+2)}$.
\end{cor}
\begin{proof}
We solve the problem via the fixed point method.
The equation we are interested in can be written
$\scal(g_{t,\phi})+v=\kappa_0$, which can be written
\begin{equation}
  \label{eq:scalphi}
P_t(v,\phi)+N(v,\phi)=\kappa - \scal(h_t)  
\end{equation}
where $N$ is the non linear term of the equation.

We are looking for a solution of the form $(v,\phi)=Q_t(f)$ and the equation reads
\begin{equation}
  \label{eq:scalf}
f = \kappa-\scal(h_t)- N_t\circ Q_t(f) =: T_t(f).  
\end{equation}
Applying the fixed point theorem, the Corollary is a consequence of
the following claim.

\smallskip
\noindent\emph{Claim}. There exists $C>0$, such that for all $t>0$
sufficiently small, the operator $T_t$ maps
the ball $\|f\|_{C^{0,\alpha}_{\delta-2}}\leq C\epsilon^2$ to itself and is $\frac
12$-contractant.

\smallskip
\noindent Let us now prove the claim.  The map $N_t(v,\phi)$ depends only on
$\phi$ so we can write it $N_t(\phi)$. Then there exists $c_2,C_2>0$ such that if
$$
\|\phi\|_{C^{4,\alpha}_2},\|\psi\|_{C^{4,\alpha}_2} \leq c_2,
$$
then
\begin{equation}
  \label{eq:contract}
\|N_t(\phi)-N_t(\psi)\|_{C^{0,\alpha}_{\delta-2}}\leq C_2(\|\phi\|_{C^{4,\alpha}_2}+\|\psi\|_{C^{4,\alpha}_2})\|\phi-\psi\|_{C^{4,\alpha}_{\delta+2}}  
\end{equation}
(cf. \cite[Lemma 19]{SzX10} and notice that the condition $\delta<0$
required  there is not needed).

By Proposition~\ref{prop:estiscal}, 
$$
\|\scal(h_t)-\kappa\|_{C^{0,\alpha}_{\delta-2}}\leq \frac 12C\epsilon^2
$$
for some constant $C>0$.  Using Proposition~\ref{prop:ri}, the bound
$\|f\|_{C^{0,\alpha}_{\delta-2}}\leq C\epsilon^2$ gives on $\phi=Q_tf$ a bound $\|\phi\|_{C^{4,\alpha}_{\delta+2}}\leq
C\epsilon^{2-\beta(\delta+2)}$ and we deduce that
$$
\|\phi\|_{C^{4,\alpha}_{2}}\leq
C\epsilon^{2+\delta-\beta(\delta+2)}=C\epsilon^{ (\delta+2)(1-\beta)}.
$$
Since $\delta+2>0$ and $1-\beta>0$, we conclude that
$\|\phi\|_{C^{4,\alpha}_{2}}= o(1)$. Using \eqref{eq:contract}, we
deduce that for $t>0$ small enough, the map $T_t$ is $1/2$-contractant
on the ball  $\|f\|_{C^{0,\alpha}_{\delta-2}}\leq C\epsilon^2$.
The map $T_t$ preserves the ball since
\begin{align*}
\|T_t(f)\|_{C^{0,\alpha}_{\delta-2}}
&\leq \|T_t(f)-T_t(0)\|_{C^{0,\alpha}_{\delta-2}} + \|T_t(0)\|_{C^{0,\alpha}_{\delta-2}} \\
&\leq \frac 12 \|f\|_{C^{0,\alpha}_{\delta-2}}+  \|\scal(h_t)-\kappa\|_{C^{0,\alpha}_{\delta-2}} \\
&\leq C\epsilon^2.
\end{align*}
\end{proof}

\begin{proof}[Proof of Proposition~\ref{prop:ri}]
  This proposition is close to \cite[Proposition 20]{SzX10}, with the
  difference that the complex structure is deformed and the sign of
  the weight $\delta+2$ is opposite to that of \cite{SzX10}. The change of
  complex structure just gives an additional error term in the
  estimates so is not a substantial change. The choice of opposite
  sign of the weight is more important, and we explain briefly how to
  deal with it.

  In this kind of problem, one obtains a right inverse for $P_t$ by
  gluing a right inverse on the ALE space $Y_t$ with a right inverse
  on the orbifold part $\cX$. The point is that we consider the weight
  $\delta+2>0$. This immediately implies that on the ALE space $Y_t$, the
  operator
  $$ L:C^{4,\alpha}_{\delta+2} \longrightarrow C^\alpha_{\delta-2} $$
  is surjective (by duality the cokernel is the kernel of $L$ in
  $C^{4,\alpha}_{-\delta-2}$, which is $0$ because $-\delta-2<0$).

  Dually, the same operator $L:C^{4,\alpha}_{\delta+2} \longrightarrow C^\alpha_{\delta-2}$ on the
  orbifold has no kernel since $\delta+2>0$ rules out the constants near
  the punctures, and $\cX$ has no holomorphic vector field. But $L$
  has a cokernel: since $L$ is selfadjoint, index theory in weighted
  spaces gives that the index of $L$ is the opposite of the number of
  punctures $k$. Define a space $\hat
  C^{4,\alpha}_{\delta+2}=C^{4,\alpha}_{\delta+2}\oplus\setR^k$ of functions of the form
  \begin{equation}
u+\sum_1^k \lambda_i \chi_i,\quad u\in C^{4,\alpha}_{\delta+2}, (\lambda_i)\in \setR^k,\label{eq:6}
  \end{equation}
  where $\chi_i$ is a cutoff function which vanishes outside a small ball
  around the puncture $x_i$. For example, we can equip the space $\hat
  C^{4,\alpha}_{\delta+2}$ with the norm
  \begin{equation}
  \| f\|_{\hat C^{4,\alpha}_{\delta+2}} = \sum_1^k |f(x_i)| + \|df\|_{C^{3,\alpha}_{\delta+1}}.\label{eq:5}
\end{equation}
  Then saying that $L$ has index $-k$ translates to the fact that
  $$ \tilde L:\setR^k\oplus\hat C^{4,\alpha}_{\delta+2} \longrightarrow C^\alpha_{\delta-2} $$
  has index $0$. Since its kernel is now reduced to the constants, its
  cokernel is also reduced to the constants.

  Now an inverse for this operator (orthogonally to the constants),
  combined with the inverses of the operators at the punctures, can be
  used to construct an approximate right inverse for $P_t$. One
  deduces that $P_t$ has a right inverse $Q_t$. The only tricky point
  is to estimate the norm $\|Q_t\|$, because of the constants at the
  punctures which appear in the space $\hat C^{4,\alpha}_{\delta+2}$. These
  constants are bounded in the space $\hat C^{4,\alpha}_{\delta+2}$, but on the
  glued manifold, they blow up in the $C^{4,\alpha}_{\delta+2}$ norm. Since they
  are cut around the radius $r=\varepsilon^\beta$, they contribute at most
  by a factor $(\varepsilon^\beta)^{-2-\delta}=\varepsilon^{-\beta(\delta+2)}$ which explains
  the norm estimate given for $Q_t$ in the statement of the
  proposition.
\end{proof}

\section{Hamiltonian stationary spheres}
\label{sec:hamilt-stat-spher}

We now construct the Hamiltonian stationary spheres and prove Theorem~\ref{th:HSS} in the case of canonical singularities. The spheres are
obtained as deformations of a Lagrangian sphere in the tangent
graviton, which is holomorphic for another complex structure in the
hyperK\"ahler family, so is Hamiltonian stationary.

\subsection{Deformation theory}
As we shall now see, the deformation theory of Hamiltonian stationary
spheres is our case is very simple.
Let us remind some basic facts about Hamiltonian stationary surfaces
in a K\"ahler 4-manifold $(X,\omega)$. A Hamiltonian stationary surface is a
Lagrangian surface which is a critical point of the area for
Hamiltonian deformations. This gives an equation that can be written
in the following way: given an embedded surface $\iota_S:S\subset X$, let
$H$ be its mean curvature vector and $\alpha=H\lrcorner \omega$, then $\alpha_S=\iota_S^*\alpha$
is a 1-form on $S$ satisfying the equation $d\alpha_S = \iota_S^*\Ric$.
Then $S$ is Hamiltonian stationary if on $S$ one has
$$ \delta\alpha_S=0. $$

In the hyperK\"ahler case, one has $\alpha_S=-d\theta$, where $\theta$ is the phase
defined from the holomorphic symplectic form $\Omega$ by
$\iota_S^*\Omega=e^{i\theta}dVol_{g_S}$, and the equation is equivalent to $\Delta\theta=0$.
If $S$ is holomorphic for another complex structure in the hyperK\"ahler
family, then $S$ is minimal so obviously Hamiltonian
stationary. Moreover, in that case, one obtains readily that the
linearization of the equation is given by $f\mapsto\Delta_S^2f$, where the
infinitesimal deformations are parameterized by a function $f$ on $S$
(the graph of $df$ in $T^*S$ giving the infinitesimal Lagrangian
deformation). The important point here is that this linearization is
automatically an isomorphism
$$ C^{k,\alpha}(S)/\setR \longrightarrow C_0^{k-4,\alpha}(S), $$
where $C_0$ denote the functions $f$ on $S$ such that $\int_S f
dVol_{g_S}=0$, and the quotient by $\setR$ is natural since constants give
trivial deformations. From this we deduce immediately the following
lemma:
\begin{lemma}\label{lem:deform-theory-hamilt}
  Suppose $S$ is a Lagrangian sphere in a hyperK\"ahler 4-manifold
  $(X,\omega)$, which is holomorphic with respect to one of the complex
  structures of $X$. If $(Y,\xi)$ is a K\"ahler manifold, sufficiently
  close to $(X,\omega)$ in $C^{2,\alpha}$ norm, such that $[S]$ remains a
  Lagrangian homology class for $\xi$ ($[\xi][S]=0$), then in $(Y,\xi)$, in
  a small $C^{3,\alpha}$ neighborhood of $S$, there exists a unique
  Hamiltonian stationary sphere $T$ such that $[T]=[S]$.
\end{lemma}
\begin{proof}
  The proof relies on two facts: if the homology class remains
  Lagrangian, it can be represented by a nearby Lagrangian surface ; and the
  linearization under Hamiltonian deformations is an isomorphism (see
  above). So the proof is standard, but we give a short argument where
  the two aspects are treated simultaneously.

  We look at maps $f:S\to X$ which are deformations of the given
  inclusion $\iota_S:S\subset X$, and consider the operator
  $$ \Phi(f,\xi)=(f^*\xi,\delta_{f^*g_\xi}(H_{f,\xi}\lrcorner \xi)), $$
  for $\xi$ a nearby K\"ahler structure (the complex structure is also
  deformed), and $H_{f,\xi}$ denotes the mean curvature vector of $f(S)$
  for the metric $\xi$. It is clear that $f(S)$ is Lagrangian stationary
  if and only if $\Phi(f)=0$.
  
  A tangent vector to the space of maps $f:S\to X$ is a section $n$ of
  the normal bundle of $S$, but we find more convenient to represent
  it by the 1-form $\alpha=n\lrcorner \omega$ on $S$. Then at the inclusion
  $\iota_S$ one has
  $$ \frac{\partial\Phi}{\partial f}(\alpha) = (d\alpha,\delta\Delta\alpha). $$
  To avoid the difference of the orders of the differential operators,
  we consider instead of $\Phi$ the operator
  $$ \Psi(f,\xi)=(f^*\xi,\Delta_{f^*g_\xi}^{-1}\delta_{f^*g_\xi}(H_{f,\xi}\lrcorner \xi)), $$
  so that
  $$ \frac{\partial\Psi}{\partial f}(\alpha) = (d\alpha,\Delta^{-1}\delta\Delta\alpha) = (d\alpha,\delta\alpha). $$
  Then $\Psi$ is a smooth operator $C^{3,\eta}\times C^{2,\eta}\to C^{2,\eta}_0\times C^{2,\eta}_0$,
  where each time the index $0$ means with zero integral over $S$ (for
  the metric $f^*g_\xi$); here we have used the hypothesis that $\Sigma$
  remains Lagrangian, so $\int_Sf^*\xi=0$. The differential $\frac{\partial\Psi}{\partial f}$ is obviously
  an isomorphism at $(f,\xi)=(\iota_S,\omega)$, since it identifies with
  $d+\delta:\Omega^1(S)\to\Omega^2(S)_0+\Omega^0(S)_0$. The result is a consequence of
  the implicit function theorem.
\end{proof}

This lemma is useful because of the following remark:
\begin{lemma}
  If $Y_\zeta$ is a gravitational instanton for some $\zeta=(\zeta_1,\zeta_2,\zeta_3)\in
  \fh\otimes\setR^3$, and $\theta$ is a positive root such that $\zeta_1\in \ker \theta$, then the
  Lagrangian homology class corresponding to $\theta$ is represented by a
  holomorphic cycle for a complex structure orthogonal to $I_1$ (and
  therefore Lagrangian for $I_1$).
\end{lemma}
\begin{proof}
  Because $\theta(\zeta_1)=0$, there exists an angle $\varphi$, such that
  $$ u=
  \begin{pmatrix}
    0 & -\cos \varphi & \sin \varphi \\ 1 & 0 & 0 \\ 0 & \sin \varphi & \cos \varphi
  \end{pmatrix} \in SO_3 $$ sends $\zeta$ to $\xi=(\xi_1,\xi_c)$ such that
  $\theta(\xi_c)=0$. By the second statement in Lemma~\ref{lem:lag-hol-cycle}, the
  homology class corresponding to $\theta$ is represented by a holomorphic cycle
  for the complex structure $u^{-1}(I_1)=-\cos \varphi I_2+\sin \varphi I_3$.
\end{proof}

Together with Lemma~\ref{lem:deform-theory-hamilt} we deduce:
\begin{cor}\label{cor:exist-uniq-ham}
  Under the hypotheses of Theorem~\ref{th:CSCK3}, fix a singular point $x\in
  \cX$, with tangent graviton $Y_{\zeta_r,\zeta_c}$. Let $\theta$ be a positive root,
  such that $\zeta\in \ker \theta$ and $\theta$ is primitive for this property, so that $\Sigma$
  is represented in $Y_{\zeta_r,\zeta_c}$ by a Hamiltonian stationary sphere $S_0$
  (Lemma~\ref{lem:lag-hol-cycle}). Finally suppose that the 2-homology
  class $\Sigma\in H_2(\cMhat_t,\setZ)$ defined by $\theta$ remains Lagrangian, that is
  $\Omega_t\cdot \Sigma=0$.

  If $\omega_t$ is the CSCK metric on $X_t$ in the class $\Omega_t$, then for $t$
  small enough, $\Sigma$ can be represented by a Lagrangian stationary sphere,
  close to $S_0$. 
\end{cor}
Here we use that $\omega_t$ is a CSCK metric only through the estimates
that it satisfies. The conclusion holds also for every metric in the
class $\Omega_t$ satisfying the same estimates; the CSCK metric is a
canonical example of such a metric.
\begin{proof}
  We blow up the metrics $\omega_t$: from Corollary~\ref{cor:CSCK}, on every
  compact, the K\"ahler metrics $\frac{\omega_t}{t^p}$ converge to the gravitational
  instanton $Y_{\zeta_r,\zeta_c}$ in $C^{2,\alpha}$ (actually in $C^\infty$). Then we apply
  Lemma~\ref{lem:deform-theory-hamilt}.
\end{proof}

\subsection{Minimizing property}
We now prove the minimizing property of the spheres that we
constructed, which is stated in Theorem~\ref{th:HSS}. The idea is that
a sequence of minimizers in the homotopy class must converge to a
minimizer in the tangent graviton $Y_{\zeta_1,\zeta_c}$. But since our
minimizer $S_0\subset Y_{\zeta_1,\zeta_c}$ is calibrated, it is unique, so the
sequence of minimizers must converge to $S_0$. But then in a
neighborhood of $S_0$ we have a uniqueness statement for our
Lagrangian stationary sphere. Of course the whole process relies
strongly on the fundamental results of Schoen-Wolfson \cite{SchWol01}.

Let us now give more details. For each $t>0$, by \cite{SchWol01} the
free homotopy class of $\Sigma$ can be represented by a $\omega_t$ Hamiltonian
stationary, weakly conformal map $s_t:S^2\to X_t$, each being area
minimizing in the homotopy class.

Of course the same map works for $\varpi_t=\frac{\omega_t}{t^p}$, which we now
choose as the metric on $\cMhat_t$. Since the area of the Hamiltonian
stationary sphere that we constructed is $O(t)$ for $\omega_t$, it is
bounded for $\varpi_t$, and so is the area of the collection $s_t$
which is not bigger.

The local geometry of $(\cMhat_t,\varpi_t)$ is controlled: indeed, on the
ALE part, $(\cMhat_t,\varpi_t)$ converges to $Y_{\zeta_1,\zeta_c}$, while on the
rest of the manifold, the curvature of $\varpi_t=\frac{\omega_t}{t^p}$ goes to
zero. Moreover the injectivity radius of $\varpi_t$ remains bounded
below. It follows that the local regularity results in \cite{SchWol01}
apply uniformly in $t$, in particular \cite[Theorem 2.8]{SchWol01}
there is a uniform H\"older bound on the $s_t$. Since $\Sigma^2\neq 0$, the
image of $s_t$ must cut $S_0$, and it follows that for $t$ small
enough, the image of $s_t$ is completely included in a bounded domain
of the ALE part.

Now we again claim that the compactness theorem \cite[Theorem
5.8]{SchWol01} applies in our context, because the geometry is
controlled. This implies that some sequence $s_{t_j}$ for $t_j\to0$
converges to a Lagrangian stationary, weakly conformal $W^{1,2}$ map
$s_0:S^2\to Y_{\zeta_1,\zeta_c}$, still representing the same homotopy
class. Since $S_0$ is calibrated, the family must identify to the
sphere $S_0\subset Y_{\zeta_1,\zeta_c}$. Since all $s_t$ are Lagrangian stationary,
by regularity \cite[Theorem 4.10]{SchWol01} the convergence is
smooth. Therefore for each $t>0$ small enough, $s_{t,1}:S^2\to \cMhat_t$
is an embedding converging to the standard embedding $S_0\to
Y_{\zeta_1,\zeta_c}$. By the uniqueness statement in 
Lemma~\ref{lem:deform-theory-hamilt} it must coincide 
with the Hamiltonian
stationary sphere that we constructed.

\section{$T$-singularities and $\QQ$-Gorenstein smoothings}
\label{sec:gor}

\subsection{CSCK metrics}
We extend our results in the setting of $\setQ$-Gorenstein smoothings. The
singularities that can appear are rational double points on one hand, on
the other hand cyclic singularities of type $\frac1{dn^2}(1,dnm-1)$, where
$n$ and $m$ are coprime integers. The last one is actually a $\setZ_n$ quotient
of the rational double point $\frac1{dn}(1,-1)$ by the action generated by
$(\xi,\xi^{dnm-1})$, for $\xi=e^{\frac{2\pi i}{dn^2}}$.

We now explain why the results of the previous sections extend. Roughly
speaking, the deformation theory of $\setQ$-Gorenstein smoothings is the $\setZ_n$
invariant part of the deformation theory for $A_{dn-1}$ singularities, and
the models we glue are the $\setZ_n$ quotients of $A_{dn-1}$ gravitational
instantons. These models (the tangent gravitons) are no more hyperK\"ahler
ALE spaces, but only K\"ahler Ricci flat ALE spaces, and are given
explicitly the Gibbons-Hawking ansatz, see \cite{Suv12}. Nevertheless
everything done in sections \ref{sec:defo}--\ref{sec:hamilt-stat-spher}
extends just by quotienting the local models by the action of $\setZ_n$. This
gives immediately Theorem~\ref{th:CSCK1} and Theorem~\ref{th:CSCK2} in the general
case of $T$ singularities.

\subsection{Hamiltonian stationary spheres}\label{sec:hamilt-stat-spher-1}
Here there are some interesting phenomenons happening in the case of
$T$ singularities. The starting point is the same: in the setting of
Theorem~\ref{th:HSS}, near a singular point $x\in \cX\hookrightarrow \cM\to\Delta$, we have,
up to a covering of group $\setZ_n$, a graviton $Y_{\zeta_1,\zeta_c}$, where $\zeta_1\in
\fh_\setR^{\setZ_n}$, $\zeta_c\in \fh_\setC^{\setZ_n}$, and the space $\cM_t$ is made by
gluing $Y_{\zeta_1,\zeta_c}/\setZ_n$ with the orbifold $\cX$. In particular, given
a large $\setZ_n$ invariant region $V\subset Y_{\zeta_1,\zeta_c}$, one can identify
$V/\setZ_n$ with some region in $U\subset\cM_t$ such that the metric
$\frac{\omega_t}{t^p}$ converges to the restriction to $V/\setZ_n$ of the ALE Ricci
flat metric. Denote this projection $p:V\to U$.

If we have a root $\theta\in \prs$, such that $\langle\theta,\dot \zeta\rangle=0$ and $\theta$ is primitive
for this property, and moreover $\langle\zeta_1(t),\theta\rangle=0$ for all $t$, then in the
local $\setZ_n$-covering $V$ the root $\theta$ represents a $p^*\Omega_t$ Lagrangian
class. Corollary~\ref{cor:exist-uniq-ham} applies as well, and $\theta$ can be
represented by a $p^*\omega_t$ Hamiltonian stationary sphere $S_t\subset V$, converging
to a sphere $S_0$ of the graviton $Y_{\zeta_1,\zeta_c}$, which is holomorphic with
respect to a complex structure, orthogonal to $I_1$. Then $p(S_t)$ is a
Hamiltonian stationary surface in $(\cM_t,\omega_t)$, but $p(S_t)$ might be not
embedded. This depends only on the model: $p(S_t)$ is embedded if $p(S_0)$
is, so we have to analyze the model.

Denote $\vartheta=\frac1n \sum_{g\in \setZ_n} g\cdot \theta$. We claim that, if $\vartheta\neq 0$, then
$p(S_0)$ is embedded, and we get a Hamiltonian stationary sphere in
the homology class $p_*\vartheta$.  This will end the proof of Theorem~\ref{th:HSS}. At the end of the section, we will also see some
examples of behaviors when $\vartheta=0$, resulting in the construction of a
$\setR P^2$ or a $S^2$ with a double point.

Fortunately the possible spheres $S_0$ are explicit in the
Gibbons-Hawking ansatz, so we merely have to check the above
claim. Therefore we remind briefly what we need \cite[\S5]{Suv12}.

We consider $k+1$ distinct points $p_0,\ldots,p_k\in \setR^3$, and the harmonic
function $V(x)=\frac12\sum_i\frac1{ |x-p_i| }$. Then $*dV$ is a closed
2-form on $\setR^3\setminus\{p_i\}$, which is furthermore integral (indeed, the
integral of $*dV$ on a small sphere around $p_i$ is $1$). Therefore
$*dV=d\eta$, where $\eta$ is the connection 1-form on the total space of a
circle bundle $L\to\setR^3\setminus \{p_i\}$. Because of the topology of $L$ near each
$p_i$, the restriction of $L$ to a small 2-sphere around $p_i$ is
diffeomorphic to a 3-sphere, and one can compactify $L$ into a smooth
manifold $M$, equipped with a projection $\pi:M\to\setR^3$, by adding just one
point above each $p_i$. Then it turns out that
$$ g = V (dx_1^2+dx_2^2+dx_3^2) + V^{-1}\eta^2 $$
is a smooth hyperK\"ahler metric on $M$, whose three complex structures
are given by
$$ I_i dx_i = V^{-1} \eta, \quad I_i dx_j=dx_k, $$
where  $(i,j,k)$ is a circular permutation of $(1,2,3)$. More
generally, for any $\xi=(\xi_1,\xi_2,\xi_3)\in S^2$, we have a complex structure
$I_\xi$ such that $I_\xi\sum \xi_idx_i = V^{-1}\eta$ and $I_\xi$ is a rotation of
angle $\frac\pi2$ in the plane $(\sum \xi_idx_i)^\perp \subset (\setR^3)^*$; the
corresponding K\"ahler form is $\omega_\xi=\sum_i \xi_i(dx_i\land \eta + V dx_j\land dx_k)$.
All the structures are invariant under the circle action.

If the segment $[p_a,p_b]$ does not contain another point $p_c$, then
$\pi^{-1}[a,b]$ is a 2-sphere, which is holomorphic for the complex
structure $I_\xi$, where $\xi=\frac{a-b}{ |a-b| }$, and Lagrangian for the
K\"ahler forms $\omega_\zeta$ for $\zeta\perp \xi$. It follows immediately that
$\pi^{-1}[a,b]$ is Hamiltonian stationary for the K\"ahler metric
$(M,I_\zeta,\omega_\zeta)$, where $\zeta\perp \xi$, and this gives our model spheres in the
case of $A_k$ singularities.

It is also interesting to describe the cohomology in this model:
define the 2-form
$$ \chi_i = df_i\land \eta + * df_i, \quad f_i(x)=\frac1{2V(x)|x-p_i|}. $$
It is easy to check that $\chi_i$ is a smooth closed antiselfdual 2-form
on $M$, and $\sum \chi_i=0$. This is the only relation and one gets the
representation of the cohomology of $M$ by harmonic forms:
\begin{equation}\label{eq:8}
  H^2(M,\setR) = \big\{ \sum_i c_i\chi_i, \sum_ic_i=0 \big\}. 
\end{equation}
(This actually describes the $L^2$ cohomology of $M$, which turns out
to be equal to the ordinary cohomology.)
The form $\chi_i$ evaluated on the 2-sphere $\pi^{-1}[p_a,p_b]$ gives
$$ \langle\chi_i,\pi^{-1}[p_a,p_b]\rangle = 2\pi(f_i(p_a)-f_i(p_b)) = 2\pi(\delta_{ia}-\delta_{ib}). $$
This formula justifies the equality (\ref{eq:8}).

In the case of $T$ singularities, we start from a $A_{dn-1}$
gravitational instanton, given from the Gibbons-Hawking ansatz with
$k+1=dn$ points. For a careful choice of the points $p_a$, there is a
free isometric action of $\setZ_n$ which is holomorphic for one of the
complex structures: denote $z=x_1+ix_2$ (this a $I_3$ holomorphic
function), we consider the action of a generator $\varpi=e^{\frac{2i\pi}n}$
of $\setZ_n$ given by
$$ \varpi \cdot z = \varpi z, \quad \varpi \cdot \theta = \varpi^m \theta, $$
where $\theta$ is the angular coordinate, and $n$ and $m$ are coprime. Of
course this action is well defined only if the configuration of the
$dn$ points is invariant under $z\mapsto\varpi z$, that is if they are organized
into a collection of $d$ regular centered $n$-polygons in planes
orthogonal to the vector $\partial_{x_3}$. This gives all the ALE Ricci flat
models that we need for singularities of class $T$. Incidentally, the
parameters are a collection of $d$ points (one in each polygon), that
is $3d$ real parameters, modulo the translation in the $\partial_{x_3}$
direction, so finally $3d-1$ parameters as expected ($d$ complex
parameters and $d-1$ real parameters). In particular, from
(\ref{eq:8}), the $\setZ_n$ invariant part of $H^2(M,\setR)$ has dimension $d-1$.

We can now describe Hamiltonian stationary spheres: as we have seen, these
are given by $\pi^{-1}[p_a,p_b]$, where $p_a-p_b\perp \partial_{x_3}$, that is by the
segments in the planes of the polygons. Here there are two cases:
\begin{itemize}
\item a segment $[p_a,p_b]$ between two points of two distinct polygons
  (which is possible only if the two polygons are in the same plane): if we
  change the point $p_b$ in the same polygon, we do not change the homology
  class in the quotient, because the sides of the polygon go to zero in the
  homology of the quotient; therefore we can choose $p_b$ to be the closest
  vertex to $p_a$, so that the images under $\setZ_n$ of the segment
  $[p_a,p_b]$ are disjoint; therefore in the quotient we still obtain a
  2-sphere with nonzero class in homology (indeed the pairing with $\frac1n
  \sum_{g\in \setZ_n}g \cdot (\chi_a-\chi_b)$ is nonzero);
\item a segment $[p_a,p_b]$ between two points of the same polygon:
  we consider only the two following cases (the images of the other spheres
  have more complicated crossings):
  \begin{itemize}
  \item $[p_a,p_b]$ is an edge of the polygon; then in the quotient, the
    points $p_a$ and $p_b$ represent the same point and we obtain an
    immersed 2-sphere with one double point;
  \item $n=2$ and $p_b=-p_a$: then the image of $\pi^{-1}[p_a,p_b]$ is an
    embedded $\setR P^2$.
  \end{itemize}
  The other segments $[p_a,p_b]$ are more complicated since for a $g\in \setZ_n$,
  the segment $g\cdot [p_a,p_b]$ might meet $[p_a,p_b]$ in an interior point,
  which means that one gets a double circle.
\end{itemize}

\section{Applications}
\label{sec:appli}
\subsection{Del Pezzo surfaces}
\label{sec:dp}
Our result is applied to produce extremal metrics on $\QQ$-Gorenstein smoothings
of singular extremal Del Pezzo surfaces with no nontrivial holomorphic
vector field.

More precisely, let $\cX$ be a normal Del Pezzo surface. If $\cX$
admits a $\QQ$-Gorenstein smoothing, then all the singularities of
$\cX$ must be of class $T$. So it is natural to assume that every
singularity of $\cX$ is of such type. Locally, we may pick a one parameter
$\QQ$-Gorenstein smoothing for each singularity of $\cX$. It turns out that
such local smoothing can always be globalized. In other words, one can
find a $\QQ$-Gorenstein smoothing $\cX\hkto\cM\to \Delta$ such that
the germs of deformations of the singularities are the one we started
with. This result is due to the fact that $H^2(\cX,T\cX)=0$, hence
there is no \emph{local to global obstruction} to deformation theory as
proved in \cite[Proposition 3.1]{HacPro10}. 

In particular, we can always construct in this way a one parameter
$\QQ$-Gorenstein smoothing $\cX\hkto\cM\to \Delta$ satisfying the non
degeneracy condition in the sense of Definition~\ref{df:degenerate}.
Suppose that $\cX$ admits an extremal K\"ahler metric and no nontrivial
holomorphic vector field (hence the metric must be CSCK) and that
$\Omega_t$ is a family of K\"ahler classes on $\cM_t$ that degenerates
toward the orbifold K\"ahler class.  Then, by
Theorem~\ref{th:CSCK1}, the smoothing $\cM_t$ admits a CSCK metric with K\"ahler
class $\Omega_t$ for all $t>0$ sufficiently small.

For instance, consider the Del Pezzo orbifold $\cX=(\CP^1\times
\CP^1)/\ZZ_4$, where the action of $\ZZ_4$ on the product is spanned
by
$$
([u_1:v_1],[u_2,v_2])\mapsto  ([iu_1:v_1],[iu_2,v_2]).
$$
The quotient contains exactly four singularities. Two of them are
$A_3$ singularities whereas the two others are of type $T$, modelled
on cyclic quotients of the form $\frac 14(1,1)$.

We construct a nondegenerate family of smoothings $\cX\hkto\cM\to \Delta$ as
explained above. Quite interestingly, a smoothing $\cM_t$ for $t\neq 0$ is not
diffeomorphic to the minimal resolution $\cXhat$ of $\cX$. In such a
situation $\cM_t$, as a smooth manifold, is obtained by removing a
neighborhood of the $-4$ exceptional spheres in $\cXhat$ and gluing back
two copies of the rational homology ball $T^*S^2/{\ZZ_2}=T^*\RP^2$. Such an
operation is known under the name of rational blowdown. Thus, $\cM_t$ is
the rational blowdown of $\cXhat$ for $t\neq 0$. As $\cXhat$ is an eight-point
iterated blow-up of $\CP^1\times\CP^1$, one can show that $\cM_t$ is a six-point
blow up of $\CP^1\times \CP^1$.

The product $\CP^1\times\CP^1$ can be endowed with a CSCK metric by
choosing a multiple of the standard Fubini-Study metric on each
factor. The group $\ZZ_4$ acts isometrically on the product, thus we
obtain a CSCK orbifold metric on $\cX$. 

Then we pick any family of K\"ahler classes $\Omega_t$ on $\cM_t$ for
$t>0$ that
degenerates toward the orbifold K\"ahler class. At the point, we are
ready to apply Theorem~\ref{th:CSCK1}. Unfortunately, $\cX$ does admit
nontrivial holomorphic vector fields. To get around this issue, we
work equivariantly, modulo an additional symmetry.
There is a $\ZZ_2$ action on $\CP^1\times\CP^1$ spanned by
$$
([u_1:v_1],[u_2:v_2])\mapsto ([v_1:u_1],[v_2:u_2]).
$$
This action descends to the quotient $\cX=(\CP^1\times\CP^1)/\ZZ_4$.
The main point is that $\cX$ does not carry any nontrivial holomorphic
vector field. Such a $\ZZ_2$-action on $\cX$ extends as a
$\ZZ_2$ fiberwise-action on the smoothing $\cM\to \Delta$. We may restrict our
attention to $\ZZ_2$-equivariant smoothing (i.e. such that $\ZZ_2$
acts trivially on $\Delta$) and $\ZZ_2$-invariant K\"ahler classes $\Omega_t$.
The proof of Theorem~\ref{th:CSCK1} can be done in this
$\ZZ_2$-equivariant context. It follows that $\cM_t$ admits a CSCK
metric with K\"ahler class $\Omega_t$ for all $t>0$ sufficiently small.

 In
conclusion $\cM_t$ carries aK\"ahler-Einstein metric for all $t$
sufficiently small. It should be pointed out that $\cM_t$ is not
diffeomorphic to the minimal resolution of $\cX$ for $t\neq 0$. In
fact $\cM_t$ is the full rational blow-down of $\cXhat$ which is
diffeomorphic to a $6$-point blow-up of $\CP^1\times \CP^1$.

Working with a K\"ahler-Einstein orbifold metric on $\cX$ and
$\Omega_t=c_1(\cM_t)$ we recover Tian's metric for certain special
Fano surfaces very close to the boundary of the moduli space, diffeomorphic to $\CP^1\times\CP^1$ blown up six times.
In addition  Theorem~\ref{th:HSS} applies in this setting. 
Let $E$ the homology class of an exceptional holomorphic sphere in the
resolution $\cXhat$. Then $E$ can be represented by a stationary
Lagrangian sphere in $\cM_t$ for $t>0$ sufficiently small with respect
to the
K\"ahler-Einstein metric. It is also possible to prove that $\cM_t$
contains two stationary Lagrangian  $\RP^2$   obtained by perturbing
the zero section of the tangent graviton $T^*\RP^2$.

Similar examples can be constructed by considering the CSCK orbifold
$\cX=(\CP^1\times \CP^1)/\ZZ_2$, an example considered by Spotti
\cite{SpX12}. In this case $\cX$ has four $A_1$ singularities and the
smoothings are diffeomorphic to the minimal resolution $\cXhat$. The
construction above provides a construction of CSCK metrics on certain
four-point blowups of $\CP^1\times \CP^1$. In particular, we can apply this to
the K\"ahler-Einstein case. Again, we prove that the $-2$ exceptional spheres
can be deformed to get stationary Lagrangian spheres in the smoothing
endowed with its K\"ahler-Einstein metric.

\subsection{Geometrically ruled CSCK orbifold surfaces}
\label{sec:geor}
A large class of geometrically ruled CSCK orbifolds can be constructed
via representation theory.
The idea is to consider an orbifold
Riemann surface $\Sbar$ and a twisted product $\cX=\Sbar\times_\rho \CP^1$
where $\rho$ is a morphism $\rho:\pi_1^{orb}(\Sbar)\to
\SU_2/\ZZ_2$, where $\pi_1^{orb}(\Sbar)$ is the orbifold fundamental
group. Here we require that if $p_i$ is a singular point of order
$q_i$ in $\Sbar$ and $l_i$ is the homotopy class of a small loop
around $p_i$, then $\rho(l_i)$ is of order $q_i$.

 If $\Sbar$ has only orbifold points of order $2$, then
$\cX$ has isolated singularities of type $A_1$. 
Suppose that $\Sbar$ carries a CSCK metric (we just have to exclude the case of
a teardrop with exactly one singularity of order $2$). Then the local
product metric provides a CSCK orbifold metric on $\cX$. 

Under the assumption that $\pi_1^{orb}(\Sbar)$ acts transitively on
$\CP^1$ via $\rho$ and that $\Sbar$ is not the football, the orbifold
$\cX$ does not carry any nontrivial holomorphic vector field \cite{RolSin05}. 

In the next section, we show that is is always possible to find a
nondegenerate $\QQ$-Gorenstein smoothing of $\cX$.
In particular Theorem~\ref{th:CSCK1} and Theorem~\ref{th:HSS} apply
and we may construct some new CSCK metrics on blownup ruled surfaces with
stationary Lagrangian spheres.

\subsection{Smoothing model for orbifold ruled surfaces}
The example of ruled orbifold $\cX\to \Sbar$ given in the previous section has
singularities that come by pair above each orbifold point in $\Sbar$.

More precisely, the local model is given by an orbifold surface
$\cY\to \setC/\ZZ_2$ 
where $\cY=(\setC\times \CP^1)/\ZZ_2$, the action of  $\ZZ_2$ is spanned by
$(u,[v:w])\mapsto (-u,[-v:w])$ and the projection map $\cY\to \setC/\ZZ_2$
is just induced by the first canonical projection.

Using the coordinates $(u,[v:w])$ on the ramified cover $\setC\times \CP^1$, we see that
$\ZZ_2$-invariant polynomials in the chart $w\neq 0$ are generated by
\begin{equation}
  \label{eq:chartw}
\left \{  \begin{array}{lll}
    x_1 &=&u^2\\
y_1 & = & (\frac vw)^2 \\
z_1 & =& \frac {uv}w
  \end{array}
\right .
\end{equation}
and the equation of $\cY$ in this chart is given by 
$$
x_1y_1=z_1^2,
$$
which is the equation of an $A_1$ singularity.
Similarly, in the chart $v\neq 0$, we have the invariant polynomials 
\begin{equation}
  \label{eq:chartv}
\left \{  \begin{array}{lll}
    x_2 &=&u^2\\
y_2 & = & (\frac wv)^2 \\
z_2 & =& \frac {uw}v
  \end{array}
\right .
\end{equation}
and the equation
$$
x_2y_2=z_2^2,
$$
giving a second $A_1$ singularity.

Putting together \eqref{eq:chartw} and  \eqref{eq:chartv}, we conclude
that $\cY$ is the  subvariety of $\setC\times \CP^2$
given by the equation
$$
x \alpha\beta = \gamma^2
$$
where $(x,[\alpha:\beta:\gamma])\in \setC\times \CP^2$, $x=x_1=x_2$,
$\frac\alpha\beta = y_1= \frac 1{y_2}$, $\frac \gamma\beta= z_1$ and
$\frac \gamma\alpha= z_2$.

We introduce  a family of deformation $\cY\hkto\cN\to \setC^2$ where
$\cN$ is the subvariety with points $(x,[\alpha:\beta:\gamma],\epsilon_1,\epsilon_2)\in
\setC\times\CP^2\times \setC^2$ given by the equation
$$
\epsilon_1\alpha^2+\epsilon_2\beta^2+x\alpha\beta=\gamma^2
$$
and the map $\cN\to \setC^2$ is induced by the canonical projection $
(x,[\alpha:\beta:\gamma],\epsilon_1,\epsilon_2)\mapsto
\epsilon=(\epsilon_1,\epsilon_2)$.
The singular locus of $\cN$
is contained in the hypersurface  $\epsilon_1\epsilon_2= 0$ and if neither
$\epsilon_1$ nor $\epsilon_2$ vanish, then the fiber $\cN_\epsilon$ is
a smooth deformation of $\cY$.

Using affine coordinates, one can check that the
parameters $\epsilon_i$ correspond to the parameters of the
semi-universal family of smoothings for the $A_1$ singularity. 
In particular, if we choose a line in $\CC^2$ distinct of the lines
$\epsilon_1=0$ or $\epsilon_2=0$, we obtain a nondegenerate family of
smoothings of $(\CC\times\CP^1)/\ZZ_2$.

We have a canonical projection $\cN\to \setC$ given by the coordinate
$x$. The restriction of this map $\cN_\epsilon \to \setC$ defines a
ruled surface over $\setC$ and  it is smooth unless
$\epsilon_1\epsilon_2=0$. One can also check that the fibers of the
ruling are all $\CP^1$ except when $x^2=4\epsilon_1\epsilon_2$
where the polynomial $\epsilon_1\alpha^2+\epsilon_2\beta^2+x\alpha\beta-\gamma^2$ splits as a
product of two polynomials of degree $1$ in $(\alpha,\beta,\gamma)$. If $\epsilon_1\epsilon_2\neq 0$,
there are two distinct fibers that consist of a union of two $\CP^1$
with normal crossing at one point. Topological considerations imply
that the curves  have selfintersection $-1$.
In conclusion, assuming $\epsilon_1\epsilon_2\neq 0$, the ruled
surface $\cN_\epsilon\to \setC$ is a two point blowup of $\setC\times\CP^1$ at
two distinct fibers. 

Let $\Delta^2_{1/2}$ be the ball of radius $1/2$ is $\setC^2$, so that
$|\epsilon_1\epsilon_2|<1/4$ for $\epsilon\in \Delta^2_{1/2}$. 
Let
$\cN'$ (resp. $\cN''$) be the restriction of $\cN$ to the domain $|x|\leq
2$  and $\epsilon \in \Delta^2_{1/2}$ (resp.  $x\in A=\{
1<|x|\leq 2\}$ and $\epsilon \in \Delta^2_{1/2}$). Accordingly,
we shall denote $\cY'=\cN'_0$ and $\cY''=\cN''_0$.

By definition,  $\cN''$ is equipped with a projection $\cN''\to
A \times \Delta^2_{1/2} $. This projection is a geometric ruling (a submersive
holomorphic map with fibers
$\CP^1$) for reducible curves appear only for $x^2=4\epsilon_1\epsilon_2$. Hence there exists a biholomorphism 
$$\phi:\cN''\to A  
\times \Delta^2_{1/2}\times \CP^1$$
which commutes with the projection maps to $A\times \Delta_{1/2}^2$. 

The restriction of the map $\phi$ induces an isomorphism $\cY''\to
A\times \CP^1$. Taking the product with the identity, we deduce an isomorphism
$$
\psi : \Delta^2_{1/2}\times \cY''\to A\times \Delta^2_{1/2}\times \CP^1.
$$
We shall the maps $\phi$ and $\psi$  to construct deformations of
orbifold ruled surfaces in the next section.

\subsection{Construction of deformations}
Let $\cX$ be a compact complex orbifold surface with a holomorphic embedding
$$j:\cY '\hkto \cX$$
where $\cY'=\cN'_0$. We shall construct a
smoothing of $\cX$ using the smoothing $\cN'$ described in the
previous section.

Let $\cX'$ be the complement in $\cX$ of the domain
$|x|\leq 1$ of $\cY'$. In particular, we have the restriction
$j:\cY''\hkto \cX'$. 
We introduce 
$$
\cM= (\Delta^2_{1/2}\times \cX')\sqcup \cN' /\sim
$$
The equivalence relation is given as follows: if $(\epsilon,m)\in
\Delta^2_{1/2}\times \cX'$ is such that $m=j(y)$ for some $y\in
\cY''$, we identify the point $(\epsilon,m)$ with $z\in \cN''\subset \cN'$ 
provided
$\psi(\epsilon,m)=\phi(z)$.

The complex variety $\cM$ endowed with the canonical maps
$\cX\hkto \cM\to\Delta^2_{1/2}$ is a flat deformation of $\cX$ and it
is non degenerate in the sense of Definition~\ref{df:degenerate}.

We deduce the following proposition by applying Theorem~\ref{th:CSCK1} and Theorem~\ref{th:HSS} to the family of smoothings
$\cX\hkto \cM\to \Delta$.
\begin{prop}
  Let $\cX=\Sbar\times_\rho \CP^1$ be a CSCK geometrically ruled
  surface with singularities of type $A_1$ and no nontrivial
  holomorphic vector fields as described in \S\ref{sec:geor}

Then there exists one parameter families of nondegenerate smoothings
$\cX\hkto\cM\to \Delta$. In particular, $\cX$ admits CSCK smoothings
with stationary Lagrangian spheres representing the vanishing
Lagrangian cycles.
\end{prop}

\end{document}